\documentclass[draft]{compositio}
\usepackage{amsmath, amscd, amssymb, mathrsfs, accents, amsfonts}
\usepackage{stmaryrd}
\usepackage{url}
\usepackage[all]{xy}
\SelectTips{cm}{}
\usepackage{longtable}
\usepackage{dsfont}
\usepackage{tikz}
\def\nicecolourscheme{\shadedraw[top color=blue!0, bottom color=blue!20]}

\definecolor{Myblue}{rgb}{0,0,0.6}
\usepackage[a4paper,colorlinks,citecolor=Myblue,linkcolor=Myblue,urlcolor=Myblue,pdfpagemode=None]{hyperref}

\newtheorem{theorem}{Theorem}[section]
\newtheorem{proposition}[theorem]{Proposition}
\newtheorem{lemma}[theorem]{Lemma}
\newtheorem{corollary}[theorem]{Corollary}
 
\theoremstyle{definition}
\newtheorem{definition}[theorem]{Definition}
\newtheorem{example}[theorem]{Example}
\newtheorem{remark}[theorem]{Remark}

\newtheorem*{theoremn}{Theorem}

\numberwithin{equation}{section}

\def\eval{\operatorname{ev}}
\def\Res{\operatorname{Res}}
\def\Inj{\operatorname{Inj}}
\def\inc{\operatorname{inc}}

\def\Coker{\operatorname{Coker}}

\def\free{\operatorname{free}}
\def\can{\operatorname{can}}
\def\ac{\operatorname{ac}}
\def\K{\mathbf{K}}

\def\D{\mathbf{D}}

\def\Hom{\operatorname{Hom}}

\def\modd{\operatorname{mod}}

\def\proj{\operatorname{proj}}

\def\Spec{\operatorname{Spec}}
\DeclareMathOperator{\Ext}{Ext}
\DeclareMathOperator{\coh}{coh}

\DeclareMathOperator{\End}{End}

\DeclareMathOperator{\str}{str}
\DeclareMathOperator{\tr}{tr}

\DeclareMathOperator{\sg}{sg}
\DeclareMathOperator{\skos}{K}
\DeclareMathOperator{\hmf}{hmf}
\DeclareMathOperator{\Perf}{Perf}
\DeclareMathOperator{\hht}{ht}
\DeclareMathOperator{\CM}{CM}

\begin{document}

\newcommand{\ud}{\mathrm{d}}
\newcommand{\cat}[1]{\mathcal{#1}}
\newcommand{\lto}{\longrightarrow}
\newcommand{\xlto}[1]{\stackrel{#1}\lto}
\newcommand{\mf}[1]{\mathfrak{#1}}
\newcommand{\md}[1]{\mathscr{#1}}
\newcommand{\dbsing}[1]{\mathbf{D}_{\sing}(#1)}
\newcommand{\perd}[1]{\overline{#1}}
\def\rdev{\mathbb{R}}
\def\dual{\vee}
\def\shom{\mathscr{H}\! om}
\def\send{\mathscr{E}\! nd}
\def\l{\,|\,}
\def\homot{\simeq}
\def\ceta{\vartheta}
\def\bs{\boldsymbol}
\def\cf{\bs{cf}}
\def\ci{\bs{ci}}
\def\ZZ{\mathbb{Z}}
\def\Ztwo{\mathbb{Z}_2}
\def\ptr{\underline{\tr}}
\def\plangle{\langle\!\langle}
\def\prangle{\rangle\!\rangle}
\def\serre{\mathbb{S}}
\def\m{\cat{M}}
\def\holn{\cat{L}}

\title{Residues and Duality for Singularity Categories of Isolated Gorenstein Singularities}
\author{Daniel Murfet}
\email{daniel.murfet@math.uni-bonn.de}
\address{Hausdorff Center for Mathematics, University of Bonn}

\begin{abstract}
We study Serre duality in the singularity category of an isolated Gorenstein singularity and find an explicit formula for the duality pairing in terms of generalised fractions and residues. For hypersurfaces we recover the residue formula of the string theorists Kapustin and Li. These results are obtained from an explicit construction of complete injective resolutions of maximal Cohen-Macaulay modules.
\end{abstract}

\maketitle

\section{Introduction}

Let $k$ be a field of characteristic zero and $\cat{T}$ a $k$-linear triangulated category. The concept of duality in $\cat{T}$ was formalised by Bondal and Kapranov \cite{Bondal89} in terms of a triangulated functor $\serre: \cat{T} \lto \cat{T}$ together with a family of nondegenerate pairings
\begin{equation}\label{eq:introserre1}
\cat{T}(Y, \serre X) \otimes_k \cat{T}(X,Y) \lto k
\end{equation}
natural in $X,Y$ and satisfying a condition involving compatibility with suspension. The motivating example is the bounded derived category of coherent sheaves on a smooth projective variety $Z$ over $k$. In this case it follows from classical Serre duality that there is a family of nondegenerate pairings for the functor $\serre = (-) \otimes^{\mathbb{L}} \omega_Z[d]$, where $d = \dim(Z)$ and $\omega_Z$ is the canonical bundle, and the pairing can be defined explicitly in terms of residues and traces \cite{ResiduesDuality}. For this reason the functor $\serre$ is referred to in general as a \emph{Serre functor}.

If we take the point of view that a triangulated category is a geometric object in its own~right, then the Serre functor and pairing play a fundamental role. For example, if $Z$ is a Calabi-Yau variety over $\mathbb{C}$, so $\omega_Z \cong \cat{O}_Z$, then the pairing in the derived category computes correlators in a quantum field theory probing the geometry of $Z$. The field theory is the B-twisted supersymmetric topological sigma model with target $Z$ on an oriented Riemann surface $M$ with boundary \cite{Witten92a,Witten92}, in which the bosonic fields are the components of maps $M \lto Z$. In the quantisation~the basic quantities of interest are correlators defined by Feynman path integrals over the space of fields. The boundary sector of the theory is described by the derived category, whose objects and morphisms correspond to branes and open strings, respectively \cite{Kontsevich, Douglas}. Remarkably, the pairing in the derived category gives the correlator of a pair of open string states, when $M$ is a disc.

In this paper we study Serre duality in a different triangulated category, the singularity category of a scheme with isolated Gorenstein singularities. Let $Z$ be a separated noetherian scheme of finite dimension over $k$, and consider the inclusion
\[
\Perf(Z) \subseteq \D^b(\coh Z)
\]
of the full subcategory of perfect complexes into the bounded derived category of coherent sheaves. Recall that a complex of coherent sheaves is \emph{perfect} if it is locally isomorphic in the derived category to a bounded complex of vector bundles. As is well-known, this subcategory $\Perf(Z)$ is dense if and only if $Z$ is regular, and for singular $Z$ this motivates the study of the Verdier quotient
\[
\D_{\sg}(Z) := \D^b(\coh Z)/\Perf(Z)\,.
\]
This quotient was studied in the affine setting by Buchweitz in an unpublished manuscript \cite{Buchweitz} and more recently in the global setting by Orlov \cite{Orlov04} in connection with string theory and mirror symmetry. In order to obtain a category whose morphism spaces are finite-dimensional we need to restrict to schemes which are \emph{Gorenstein} \cite{Orlov04,Avramov2}, which means that the local rings $\cat{O}_{Z,x}$ have finite injective dimension as modules over themselves, for every $x \in Z$. 

In this paper we also restrict to singularities which are isolated, in which case $\D_{\sg}(Z)$ decomposes, up to direct summands, as a direct sum of categories $\D_{\sg}(\Spec(\cat{O}_{Z,x}))$ as $x$ ranges over the singular locus, and it therefore suffices to consider this local situation.

So let $(R,\mf{m},k)$ be a local Gorenstein $k$-algebra of Krull dimension $d$ with an isolated singularity, by which we mean that $R_{\mf{p}}$ is a regular local ring for every non-maximal prime ideal $\mf{p} \subseteq R$. The object of interest in this paper is the triangulated category
\[
\D_{\sg}(R) := \D^b(\modd R)/\K^b(\proj R)\,.
\]
The morphisms in this category are defined via a calculus of fractions, and it will be more convenient to work in a category of ``resolutions'', namely the homotopy category
\[
\cat{T} := \K_{\ac}(\free R)
\]
of acyclic complexes of finite free $R$-modules. The functor
\begin{gather*}
\cat{T} \lto \D_{\sg}(R)\\
X \mapsto \Coker(X^{-1} \lto X^0)
\end{gather*}
is an equivalence of triangulated categories \cite[(4.4.1)]{Buchweitz}, so we are justified in working exclusively in the category $\cat{T}$. We note that cokernels of the differentials in complexes in $\cat{T}$ are precisely the \emph{maximal Cohen-Macaulay $R$-modules}, and the stable category $\underline{\CM}(R)$ of these modules is equivalent to $\cat{T}$. The stable category is a classical object of singularity theory; for example, results of Kn\"orrer \cite{Knorrer} and Buchweitz-Greuel-Schreyer \cite{BGS87} characterise the simple hypersurface singularities in terms of the structure of $\underline{\CM}(R)$.

A fundamental theorem of Auslander \cite{Auslander78} states that $\cat{T}$ has a Serre functor $\serre = (-)[d-1]$. There is another proof due to Buchweitz, who points out \cite[\S 7.7]{Buchweitz} that it would be interesting to have a closed formula for the corresponding pairing. There was little progress on this question until 2003, when the mathematical physicists Kapustin and Li derived a formula for the pairing in the singularity category of a hypersurface $\{ W = 0 \} \subseteq \mathbb{C}^n$ with an isolated singularity. More precisely, they found a formula for disc correlators in the B-twisted supersymmetric Landau-Ginzburg model with target space $\mathbb{C}^n$ and potential $W$ \cite{KapustinLi, Herbst05}. The singularity category of the hypersurface appears as a category of boundary conditions in this model, so their formula for the disc correlator gave a strong candidate for the duality pairing. However, it remained an open question how to prove that this candidate pairing was actually nondegenerate. 
\\

In the rest of this introduction we state our general formula for a nondegenerate pairing on the morphism spaces of $\cat{T}$, and explain how this specialises to the Kapustin-Li formula for hypersurfaces.

Fix a complex $X \in \cat{T}$ with differential $\partial$. The punctured spectrum $U = \Spec(R) \setminus \{ \mf{m} \}$ is regular by hypothesis, so the restriction $X|_U$ is contractible. We can therefore choose a cover of $U$ by open sets $D(t_1),\ldots,D(t_d)$, or what is the same, choose a regular sequence $\bs{t} = (t_1,\ldots,t_d)$ in the maximal ideal, such that $t_i$ acts null-homotopically on $X$ for $1 \le i \le d$. Choose a homotopy $\lambda_i$ on $X$ with
\[
\lambda_i \circ \partial + \partial \circ \lambda_i = t_i \cdot 1_X\,.
\]
Let $(-1)^F$ be the grading operator on $X$ which sends a homogeneous element $x \in X$ to $(-1)^{|x|} x$. Given $\alpha \in \cat{T}(X, X[d-1])$ we consider the following degree zero $R$-linear operator on $X$
\begin{equation}\label{eq:defnptr0}
\holn_\alpha := (-1)^F \alpha \circ \lambda_1 \cdots \lambda_d \circ \partial\,.
\end{equation}
The $R$-module $X$ is certainly not finitely generated, so $\holn_\alpha$ does not have a trace in the usual sense. But in each degree $i$ we can take the trace of the endomorphism $\holn^i_{\alpha}$ of $X^i$, and the class
\begin{equation}\label{eq:defnptr}
\plangle \alpha \prangle :=  (-1)^{\binom{d+1}{2}}\begin{bmatrix} \tr( \holn^i_\alpha ) \\ t_1, \ldots, t_d \end{bmatrix} \in H^d_{\mf{m}}(R) 
\end{equation}
in local cohomology is independent of all choices: the integer $i$, the system of parameters $\bs{t}$, and the null-homotopies $\lambda_j$. Here we use the notation of generalised fractions, which is recalled in Section \ref{section:cinjres}. In short: there is an isomorphism of local cohomology with \v{C}ech cohomology $H^d_{\mf{m}}(R) \cong \check{H}^{d-1}(U, \cat{O}_U)$ for $d > 0$ which identifies $\plangle \alpha \prangle$, up to a sign, with the \v{C}ech cocycle $[\tr(\holn^i_\alpha)/(t_1 \cdots t_d)]$. 

There is a $k$-linear map $\zeta: H^d_{\mf{m}}(R) \lto k$ such that composing with $\zeta$ defines an isomorphism
\[
\Hom_R(M, H^d_{\mf{m}}(R)) \lto \Hom_k(M,k)
\]
for any finite-length $R$-module $M$. The value of $\zeta$ on a generalised fraction should be thought of as a residue; indeed, if $R$ is given as a quotient of a power series ring by a regular sequence then $\zeta$ can be defined explicitly in terms of residues over the power series ring.

With this notation, our main theorem is the following:

\begin{theoremn} There is a nondegenerate pairing
\[
\langle -, - \rangle: \cat{T}(Y,X[d-1]) \otimes_k \cat{T}(X,Y) \lto k
\]
natural in both variables and compatible with suspension, defined by
\[
\langle \psi, \phi \rangle = \zeta \plangle \psi \circ \phi \prangle = (-1)^{\binom{d+1}{2}} \zeta\! \begin{bmatrix} \tr\left( \psi \circ \phi \circ \lambda_1 \cdots \lambda_d \circ \partial \right)^0 \\ t_1, \ldots, t_d \end{bmatrix}\,.
\]
\end{theoremn}

Let us now consider hypersurfaces, where $\langle \psi, \phi \rangle$ agrees with the pairing of Kapustin and Li.
\\

\textbf{Matrix factorisations.} Suppose now that $R$ is a hypersurface singularity, that is, $R = S/(W)$ where $S = k[[x_1,\ldots,x_n]]$ and $W$ is a polynomial whose zero locus in $\mathbb{A}^n_k$ has an isolated singularity at the origin. In this case it is a theorem of Eisenbud \cite{Eisenbud80} that every acyclic complex of finite free $R$-modules is two-periodic (up to homotopy equivalence) and there is an alternative presentation of the category $\cat{T}$ which makes use of this additional symmetry. 

A \emph{matrix factorisation} of $W$ over $S$ is a $\Ztwo$-graded free $S$-module of finite rank $X = X^0 \oplus X^1$ together with an $S$-linear map $d: X \lto X$ of degree one with $d^2 = W \cdot 1_X$. The map $d$ is referred to as the differential. A morphism of matrix factorisations is an $S$-linear map of degree zero which commutes with the differentials. There is an obvious notion of homotopy, using which we define the homotopy category $\hmf(S,W)$ of matrix factorisations of $W$ over $S$. If $X$ is a matrix factorisation then the differential on $X \otimes_S R$ actually squares to zero, and by the \emph{periodification} of this $\Ztwo$-graded complex we mean the $\ZZ$-graded complex of $R$-modules
\[
\xymatrix{
\cdots \ar[r] & X^1 \otimes_S R \ar[r]^-{d^1 \otimes 1} & X^0 \otimes_S R \ar[r]^-{d^0 \otimes 1} & X^1 \otimes_S R \ar[r]^-{d^1 \otimes 1} & X^0 \otimes_S R \ar[r] & \cdots
}
\]
with $X^0 \otimes_S R$ in degree zero. Abusing notation, we denote this complex by $\perd{X}$. This construction is functorial, and one can show that $\perd{X}$ is an acyclic complex of finite free $R$-modules and that this defines an equivalence of triangulated categories $\hmf(S, W) \lto \cat{T}$. 

Let us write $\m = \hmf(S,W)$. Induced by this equivalence and the duality structure given on $\cat{T}$ above, there is for $X,Y \in \m$ a nondegenerate pairing
\begin{equation}
\langle -, - \rangle: \m(Y, X[n]) \otimes_k \m(X,Y) \lto k
\end{equation}
natural in both variables and compatible with suspension. By choosing the system of parameters $\bs{t}$ and the null-homotopies $\lambda_i$ appropriately, the pairing can be put in the form
\begin{equation}\label{eq:klformulafinalfinal}
\langle \psi, \phi \rangle = \frac{1}{n!}(-1)^{\binom{n-1}{2}} \Res_{S/k}\!\begin{bmatrix} \str_S\!\left(\psi \circ \phi \circ \ud_{S/k}(d_X)^{\wedge n}\right) \\ \partial_1 W, \ldots, \partial_n W \end{bmatrix}\,.
\end{equation}
In this paper we use the residue symbols of Grothendieck \cite{ResiduesDuality} which are defined algebraically, but for $k = \mathbb{C}$ these residues agree with the usual analytic residues defined by integration and, modulo the sign, (\ref{eq:klformulafinalfinal}) is the pairing derived by Kapustin and Li \cite{KapustinLi}. Recently Ed Segal \cite{Segal09} found a derivation of this formula via Hochschild homology of curved dg-algebras, and Carqueville \cite{Carqueville} gave another derivation using the theory of minimal models for $A_\infty$-categories, but neither of these approaches gives a proof of nondegeneracy. We note that the Kapustin-Li pairing has been used in connection with Khovanov and Rozansky's $\mathfrak{sl}(N)$ link homology \cite{Mackaay09}.

We conclude with a sketch of how this pairing arises in mathematical physics, following \cite{KapustinLi} and \cite{Herbst05}. In the B-type supersymmetric topological Landau-Ginzburg model with worldsheet $M$, flat target space $Z = \mathbb{C}^n$ and potential $W \in \mathbb{C}[x^1,\ldots,x^n]$, the bosonic fields are the components\footnote{We use $\mu$ rather than the more standard $\phi$ to avoid a clash of notation with the above.} $\mu^i, \mu^{\bar\imath}$ of maps $\mu: M \lto Z$. For simplicity we are going to assume that $M$ is a disc with boundary $C = \partial M$ and we will omit the fermionic fields $\eta, \theta$ and $\rho$ from the discussion. 

Following a suggestion of Kontsevich it was explained by Kapustin and Li \cite{Kapustin03} how some matrix factorisations of $W$ appear naturally as boundary conditions in this model, and it was later understood how to introduce arbitrary matrix factorisations \cite{Laz03}. Let us take this as our starting point, and fix a matrix factorisation $X$ of $W$ together with a connection. A vector bundle together with a connection is a gauge field, and in the approach of \cite{Laz03} (following \cite{Witten92}) one couples the bulk theory to this gauge field by introducing a boundary term $\mathcal{U}$ in the partition function
\begin{equation}\label{eq:partitionfunction}
Z = \int \mathcal{D} \Psi \, \exp(- \widetilde{S}_{\text{bulk}}) \, \mathcal{U}\,.
\end{equation}
Here $Z$ is a path integral over the space of all fields $\Psi$, and $\widetilde{S}_{\text{bulk}} = \widetilde{S}_{\text{bulk}}(\Psi)$ is the bulk action. To define $\mathcal{U}$, consider pulling back the vector bundle $X$ with its connection $A$ via $\mu$ to the boundary circle $C$, and taking the holonomy of the connection around this loop; then $\mathcal{U}$ is the supertrace
\[
\mathcal{U} = \str\left( P \exp\left( - \oint_C \ud \tau \, M \right) \right)
\]
where $M$ is a matrix built from the differential $d_X$ on $X$ and connection $A$ (see \cite[(4.17)]{Herbst05}), $\ud \tau$ is the length element along $C$ and $P \exp$ is the path-ordered exponential. Note that the entries of $d_X, A$ are polynomials in the $x^i$, so pulling back to $C$ amounts to substituting $\mu^i + i \mu^{\bar\imath}$ for $x^i$.

The disc correlator of interest to us is defined by an integral similar to (\ref{eq:partitionfunction}). Let $\alpha$ be a degree $d-1$ endomorphism of $X$ in the category of matrix factorisations. Pulled back to $C$ this is a matrix of polynomials in the $\mu^i, \mu^{\bar\imath}$. In the quantum field theory the fields $\mu^i, \mu^{\bar\imath}$ are promoted to operators $\mu^i(\tau), \mu^{\bar\imath}(\tau)$ labeled with points $\tau$ of the worldsheet. Given $\tau \in C$ these operators are substituted into $\alpha$ in order to define the corresponding boundary observable as the matrix of operators
\[
\cat{O}_\alpha(\tau) = \alpha( \mu(\tau) )\,.
\]
The disc correlator $\langle \cat{O}_\alpha(\tau) \rangle$ is defined by
\begin{equation}\label{eq:picture_kaplicorrelator}
\left\langle
\begin{tikzpicture}[scale=0.6,baseline]
\draw[very thick] (0,0) circle (2cm);
\nicecolourscheme (0,0) circle (2cm);
\draw (0,0) node {{$M$}};
\draw (180:2.7cm) node {{$\cat{O}_\alpha$}};
\filldraw [black] (180:2cm) circle (3pt);
\draw (2.6,0) node {{$X$}};
\end{tikzpicture}
\right\rangle = \int \mathcal{D} \Psi \, \exp(- \widetilde{S}_{\text{bulk}}) \, \str\!\big[ H(\tau) \cat{O}_\alpha(\tau) \big]
\end{equation}
where $H(\tau)$ is the superholonomy operator. After some careful argument, one shows that this path integral localises to a finite-dimensional integral on the target space $Z$, and this integral is precisely the residue (\ref{eq:klformulafinalfinal}) with $\alpha = \psi \circ \phi$. Note that the field theory is topological, so the correlator does not depend on the position $\tau$ where $\cat{O}_\alpha$ is inserted.
\\

\textbf{Outline.} We begin in Section \ref{section:outline} with a proof of Auslander's duality in the singularity category which is adapted to finding an explicit formula for the corresponding nondegenerate pairing. In Section \ref{section:cinjres} we construct the explicit complete injective resolutions required to actually make the pairing explicit. The main theorem of the paper, quoted above, is proven in Section \ref{section:duality}. Finally in Section \ref{section:mfs} we specialise to hypersurfaces and discuss the Kapustin-Li formula.
\\

\emph{Acknowledgements.} It is a pleasure to thank Bernhard Keller, Amnon Neeman, Henning Krause, Andreas Recknagel, Ragnar Buchweitz, Joseph Lipman, Tony Pantev, Jesse Burke and Hailong Dao for enlightening discussion on the subject of this paper. I am grateful to Nils Carqueville for sharing his Singular algorithms and many detailed comments on the draft, to Srikanth Iyengar for helping with the proof of Lemma \ref{lemma:special_reg_seq_exists}, and to Igor Burban for his encouragement and for introducing me to the paper \cite{KapustinLi}. I owe special thanks to Tobias Dyckerhoff, who pointed out to me that the perturbation arguments used here could be phrased in terms of the standard homological perturbation lemma; this has allowed the current version of the paper to be much shorter!

\section{Preliminaries}\label{section:prelim}

Let $k$ be a field and $\cat{T}$ a $k$-linear triangulated category. A \emph{Serre functor} in $\cat{T}$ is a $k$-linear triangulated functor $\serre: \cat{T} \lto \cat{T}$ together with a family of $k$-linear isomorphisms
\[
\Lambda_{X,Y}: \cat{T}(Y, \serre X) \lto \Hom_k(\cat{T}(X, Y),k)
\]
which are natural in $X,Y$ and compatible with suspension, by which we mean that the diagram
\[
\xymatrix@C+3.5pc@R+0.5pc{
\cat{T}(Y[-1], \serre X) \ar[r]_-{\cong}^-{\Lambda_{X, Y[-1]}} & \Hom_k(\cat{T}(X, Y[-1]),k)\\
\cat{T}(Y, (\serre X)[1]) \ar[u]_{\cong}\\
\cat{T}(Y, \serre (X[1])) \ar[u]_{\cong} \ar[r]^-{\cong}_-{\Lambda_{X[1], Y}} & \Hom_k(\cat{T}(X[1], Y),k) \ar[uu]^{\cong}\,.
}
\]
commutes. Alternatively, we can present $\Lambda_{X,Y}$ as a family of nondegenerate pairings which satisfy conditions expressing the same naturality and compatibility with suspension.

If we say that a local ring $(R,\mf{m},k)$ is a $k$-algebra, it will be implicit that $k \lto R \lto R/\mf{m} = k$ is the identity. Unless specified otherwise all tensor products are $R$-linear.

For background on matrix factorisations see \cite{Yoshino90,Burban08} for commutative algebra, especially local cohomology, we recommend \cite{Herzog}, and for triangulated categories see \cite{Verdier96,NeemanBook}.

\section{Auslander's Duality}\label{section:outline}

In this section we give a new proof of Auslander's duality in the singularity category of a Gorenstein isolated singularity. The argument is predicated on the fundamental fact, proved in the next section, that we can explicitly construct complete injective resolutions. In organising the proof this way we hope to convince the reader of the utility of complete injective resolutions before getting down to the hard work of the construction in the next section. The reader who wants to read the full details in linear order should proceed to read Section \ref{section:cinjres} and then return here.

The setting is more general than the one adopted in the introduction: throughout $(R,\mf{m},k)$ is a local Gorenstein ring of Krull dimension $d$ with an isolated singularity. We do not assume that $R$ is a $k$-algebra, but we will remark on the additional features in this case as we go along. 

We have already introduced the equivalence of triangulated categories 
\begin{equation}\label{eq:kactodsing}
\cat{T} \lto \D_{\sg}(X)
\end{equation}
due to Buchweitz \cite[(4.4.1)]{Buchweitz}, where
\begin{equation}
\cat{T} = \K_{\ac}(\free R)\,.
\end{equation}
This equivalence sends $T$ to the cokernel of the differential $T^{-1} \lto T^0$. The inverse functor sends a module $M$, viewed as an object of the category $\D_{\sg}(R)$, to an acyclic complex of free modules usually called the complete free resolution. This complex will play a role analogous to that played by the projective resolution in the derived category: morphisms in $\D_{\sg}(R)$ are modeled by homotopy equivalence classes of maps between complete resolutions.

For convenience we choose to formalise the definition in the following way:

\begin{definition}\label{defn:cfreeres} A \emph{complete free resolution} of a finitely generated $R$-module $M$ is an acyclic complex of finite free $R$-modules $T$ together with a morphism $\rho: T \lto M$ of complexes of $R$-modules with the following properties:
\begin{itemize}
\item[(i)] $\rho$ \emph{is universal}: for any acyclic complex $F$ of free $R$-modules the induced map
\[
\Hom_R(F, \rho): \Hom_R(F, T) \lto \Hom_R(F, M)
\]
is a quasi-isomorphism.
\item[(ii)] $\rho$ \emph{truncates to a resolution}: the sequence $T^{-1} \lto T^0 \xlto{\rho^0} M \lto 0$ is exact.
\end{itemize}
The complete free resolution, if it exists, is unique up to homotopy equivalence and is denoted $\cf M$.
\end{definition}

\begin{remark}\label{remark:functor_cprojres} The complete free resolution is functorial: given a morphism $\alpha: M \lto N$ of finitely generated $R$-modules with complete free resolutions $\cf M \lto M$ and $\cf N \lto N$, there is a unique morphism $\cf(\alpha): \cf M \lto \cf N$ in $\K(R)$ making the diagram
\[
\xymatrix@C+2pc{
\cf M \ar[d]_{\cf(\alpha)}\ar[r] & M \ar[d]^\alpha\\
\cf N \ar[r] & N
}
\]
commute up to homotopy.
\end{remark}

It is easy to see that every object of $\D_{\sg}(R)$ is isomorphic to a finitely generated $R$-module~$M$ with the property that $\Ext^i_R(k,M) = 0$ for $i < d$. Equivalently, there is a sequence in $\mf{m}$ of length $d$ which is regular on $M$. Such modules are called \emph{maximal Cohen-Macaulay modules} (or CM modules, for short). Below we sketch how to construct the complete free resolution of such a module; if a CM module has finite projective dimension then it is projective, and in this case it will be clear that the complete free resolution is contractible. It follows that there is a functor $\cf(-)$ from $\D_{\sg}(R)$ to $\cat{T}$ which is, by definition, inverse to (\ref{eq:kactodsing}).

For complete details of the following construction see \cite[XII.3]{CartanEilenberg} and \cite{Buchweitz,Christensen00}. Let $M$ be a CM $R$-module. We begin with free resolutions of $M$ and its dual $\Hom_R(M,R)$:
\begin{gather*}
\cdots \lto F^{-1} \lto F^0 \lto M \lto 0\\
\cdots \lto Q^{-1} \lto Q^0 \lto \Hom_R(M,R) \lto 0\,.
\end{gather*}
Then one splices $F$ with $\Hom_R(Q,R)$ using the morphism
\begin{equation}\label{eq:constructed_cproj_ress}
\xymatrix{
F^0 \ar[r] & M \ar[r] & \Hom_R(\Hom_R(M,R),R) \ar[r] & \Hom_R(Q^0,R)
}
\end{equation}
to obtain a complex $T$ of finite free $R$-modules with $F^0$ in degree zero. By construction $M$ is the cokernel of the differential $T^{-1} \lto T^0$ and using the fact that $M$ is CM one argues that $T$ is acyclic. This is the desired complete free resolution of $M$.

\begin{lemma}\label{lemma:source_coprojres} Given $T \in \K_{\ac}(\free R)$ the module $M = \Coker(T^{-1} \lto T^0)$ is CM and the canonical map $T^0 \lto M$, viewed as a morphism of complexes $T \lto M$, is a complete free resolution of $M$.
\end{lemma}
\begin{proof}
This is the content of \cite[Lemma 3.6]{Jorgensen07}. The point is that brutally truncating $T$ in degrees $\le 0$ gives a free resolution $F$ of $M$, and the mapping cone of the truncation morphism $T \lto F$ is homotopy equivalent to a bounded below complex $K$ of free modules. If $T'$ is an acyclic complex of free modules then one checks that $\Hom_R(T',K)$ must be acyclic, from which the necessary universal property of $T \lto M$ follows.
\end{proof}

Next we explain why the morphism spaces in $\cat{T}$ have finite-length. By hypothesis, for each non-maximal prime ideal $\mf{p}$ the ring $R_{\mf{p}}$ is regular, and therefore every finitely generated $R_{\mf{p}}$-module has projective dimension $\le d$. If $T$ is an acyclic complex of finite free $R_{\mf{p}}$-modules then the brutal truncation in degrees $\le n$ serves as a free resolution of some module, and hence the modules $Z^i(T)$ are projective for $i \le n - d + 1$. Since $n$ is arbitrary, $T$ is contractible.

\begin{lemma}\label{lemma:homs_finite_length} For $X,Y \in \cat{T}$ the $R$-module $\cat{T}(X,Y)$ has finite length.
\end{lemma}
\begin{proof}
For $\mf{p}$ a non-maximal prime $\cat{T}(X,Y)_\mf{p} = H^0 \Hom_{R_\mf{p}}(X_{\mf{p}}, N_{\mf{p}}) = 0$ where $N = \Coker(\partial^{-1}_Y)$.
\end{proof}

If $R$ is a $k$-algebra then it follows that $\cat{T}(X,Y)$ is a finite-dimensional $k$-vector space, and duality in $\cat{T}$ can then formulated in terms of the dual $\Hom_k(\cat{T}(X,Y),k)$. In general we proceed differently, using the local cohomology module
\begin{equation}\label{eq:defn_localcohomology}
H^d_{\mf{m}}(R) := \varinjlim_i \Ext^d_R(R/\mf{m}^i, R)\,.
\end{equation}
Because $R$ is Gorenstein this is an injective envelope of $k$, and the Matlis dual $\Hom_R(-,H^d_{\mf{m}}(R))$ is therefore exact. If $R$ is a $k$-algebra then the Matlis dual agrees with the usual dual on the category of finite-length $R$-modules, that is, for any finite-length $R$-module $M$ there is a natural isomorphism
\begin{equation}\label{eq:firstoccurdualisingpair}
\Hom_R(M,H^d_{\mf{m}}(R)) \cong \Hom_k(M,k)\,.
\end{equation} 
We will have more to say about this isomorphism and its relation to residues in Section \ref{section:dualisingpairsyo}, but for now let us proceed to establish duality in $\cat{T}$ using the functor $\Hom_R(-,H^d_{\mf{m}}(R))$. 

As has already been mentioned, the next theorem gives a new proof of a result originally due to Auslander \cite[Proposition 8.8 in Ch.1 and Proposition 1.3 in Ch.3]{Auslander78}, of which there is another proof by Buchweitz \cite[Theorem 7.7.5]{Buchweitz}.

\begin{theorem}\label{theorem:ausdual} Given $X,Y \in \cat{T}$ there is an isomorphism of $R$-modules
\begin{equation}\label{eq:ausdual1}
\cat{T}(Y, X[d-1]) \cong \Hom_R(\cat{T}(X,Y),H^d_{\mf{m}}(R))
\end{equation}
which is natural in both variables and compatible with suspension, in the sense of Section \ref{section:prelim}.
\end{theorem}
\begin{proof}
Let us set $N = \Coker(Y^{-1} \lto Y^0)$ so that the canonical map $Y \lto N$ is a complete free resolution. There is a natural quasi-isomorphism ($\simeq$ denotes quasi-isomorphisms and $X^\dual$ is the dual complex $\Hom_R(X,R)$)
\begin{align*}
\Hom_R(\Hom_R(X,Y),H^d_{\mf{m}}(R)) &\simeq \Hom_R(\Hom_R(X, N),H^d_{\mf{m}}(R))\\
&\cong \Hom_R( X^\dual \otimes N, H^d_{\mf{m}}(R))\\
&\cong \Hom_R( N \otimes X^\dual, H^d_{\mf{m}}(R))\\
&\cong \Hom_R( N, \Hom_R( X^\dual, H^d_{\mf{m}}(R)))\\
&\cong \Hom_R( N, X^{\dual\dual} \otimes H^d_{\mf{m}}(R))\\
&\cong \Hom_R( N, X \otimes H^d_{\mf{m}}(R))\,.
\end{align*}
At this point we make use of two facts whose proofs will be given in the next section. The first is that $X \otimes H^d_{\mf{m}}(R)$ is an acyclic complex of injectives, and the second is that for any acyclic complex of injectives $I$ there is a quasi-isomorphism
\[
\Hom_R(N, I) \lto \Hom_R(Y \otimes H^d_{\mf{m}}(R)[1-d], I)
\]
induced by a special morphism of complexes $N \lto Y \otimes H^d_{\mf{m}}(R)[1-d]$ called the \emph{complete injective resolution} of $N$. Taking this as a given, we may continue
\begin{align*}
\Hom_R(\Hom_R(X, Y),H^d_{\mf{m}}(R)) &\simeq \Hom_R( Y \otimes H^d_{\mf{m}}(R)[1-d], X \otimes H^d_{\mf{m}}(R))\\
&\cong \Hom_R( Y [1-d], \Hom_R(H^d_{\mf{m}}(R), X \otimes H^d_{\mf{m}}(R)))\\
&\cong \Hom_R( Y[1-d], X \otimes \End_R( H^d_{\mf{m}}(R) ) )\\
&\cong \Hom_R( Y[1-d], X \otimes \widehat{R} )\\
&\simeq \Hom_R( Y[1-d], X ) \otimes \widehat{R}\\
&\simeq \Hom_R( Y, X[d-1] )\,.
\end{align*}
Here $\widehat{R}$ is the $\mf{m}$-adic completion of $R$, and we have used a theorem of Matlis which states that $H^d_{\mf{m}}(R)$ is naturally a $\widehat{R}$-module and sending $r \in \widehat{R}$ to multiplication by $r$ on $H^d_{\mf{m}}(R)$ gives an isomorphism $\widehat{R} \cong \End_R( H^d_{\mf{m}}(R) )$. The last line follows from the fact that the Hom spaces in $\cat{T}$ have finite-length, and taking $H^0$ in the above gives the desired isomorphism in (\ref{eq:ausdual1}).

Naturality of this isomorphism is clear from the construction, and compatibility with suspension can be checked directly, but this is tedious; we will see a better proof in Lemma \ref{lemma:ptrproperties} below.
\end{proof}

If $R$ is a $k$-algebra then combining the theorem with (\ref{eq:firstoccurdualisingpair}) yields a natural isomorphism
\begin{equation}
\cat{T}(Y, X[d-1]) \cong \Hom_k( \cat{T}(X,Y), k )\,.
\end{equation}
Observe that every step in the proof of the theorem makes use of explicit standard isomorphisms, with the exception of the step involving the complete injective resolution, and if $R$ is a $k$-algebra the isomorphism (\ref{eq:firstoccurdualisingpair}). If we want explicit formulas we must therefore understand these maps. 

\section{Complete injective resolutions}\label{section:cinjres}

Let $(R,\mf{m},k)$ be a local Gorenstein ring of Krull dimension $d$ with an isolated singularity.

\begin{definition}\label{defn:cinjres} A \emph{complete injective resolution} of an $R$-module $M$ is an acyclic complex of injective $R$-modules $I$ together with a morphism of complexes $\vartheta: M \lto I$ which \emph{universal}, in the sense that for any acyclic complex $J$ of injective $R$-modules the induced map
\[
\Hom_R(\vartheta, J): \Hom_R(I, J) \lto \Hom_R(M, J)
\]
is a quasi-isomorphism \cite{Enochs}.

The complete injective resolution, if it exists, is unique up to homotopy equivalence and denoted $\ci M$. If $\alpha: M \lto N$ is a morphism of $R$-modules the induced morphism on the complete injective resolutions is denoted $\ci(\alpha)$.
\end{definition}

Let $M$ be a CM $R$-module with complete free resolution $\rho: T \lto M$. The differential on $T$ will be denoted $\partial$. Because the singularity of $R$ is isolated the complex $T_{\mf{p}}$ is contractible for every non-maximal prime $\mf{p}$. We are going to construct a complete injective resolution of $M$, essentially by making explicit the fact that $T$ is ``supported'' on the closed point, and for this we use the stable Koszul complex of local cohomology.

Let $\bs{t} = (t_1,\ldots,t_d)$ be a system of parameters for $R$, that is, a sequence of length $d$ generating an $\mf{m}$-primary ideal in $R$. Since $R$ is Gorenstein, this is the same as a regular sequence of length $d$ in $\mf{m}$ or a sequence $\bs{t}$ such that the open sets $D(t_i)$ cover the punctured spectrum $U = \Spec(R) \setminus \{ \mf{m} \}$. The \emph{stable Koszul complex} $\skos_\infty = \skos_\infty(\bs{t})$ is the tensor product
\begin{equation}
\skos_\infty := \bigotimes_{i=1}^d \left( \xymatrix{ \underline{R} \ar[r]^-{\can} & R[t_i^{-1}] \theta_i } \right)\,,
\end{equation}
where the underline indicates cohomological degree zero. Here the $\theta_i$ are formal variables of degree one, introduced to keep track of the grading. If we adopt the convention that the $\theta_i$'s anticommute and that $\theta_i^2 = 0$, then $\skos_\infty$ is the $\mathbb{Z}$-graded $R$-module
\[
\skos_\infty = \bigoplus_{i_1 < \cdots < i_p} R[t_{i_1}^{-1}, \ldots, t_{i_p}^{-1}] \theta_{i_1} \cdots \theta_{i_p}
\]
with the differential $\delta$ given by left multiplication with $\sum_i \theta_i$. If a prime ideal $\mf{p}$ is non-maximal~then $\mf{p} \in D(t_i)$ for some $i$, so $R_\mf{p} \lto R_\mf{p}[t_i^{-1}]$ is an isomorphism and $(\skos_\infty)_\mf{p}$ is contractible. Hence the cohomology of $\skos_\infty$ is supported on the closed point $\mf{m}$, and moreover the projection
\[
\varepsilon: \skos_\infty \lto R
\]
is the universal morphism in the unbounded derived category $\D(R)$ from a complex supported on the closed point to $R$. More precisely, if $Z$ is a complex whose cohomology is supported on $\mf{m}$ then any morphism $Z \lto R$ in the derived category factors uniquely through $\varepsilon$.

The local cohomology of $R$ is defined by  $H^i_{\mf{m}}(R) := H^i(\skos_\infty)$. The cone of the projection $\varepsilon$ is a shift of the \v{C}ech complex of $U$ and it follows that $H^d_{\mf{m}}(R) \cong H^{d-1}(U,\cat{O}_U)$ for $d > 0$. Since $R$ is Gorenstein we have $H^i_{\mf{m}}(R) = 0$ for $i < d$ and therefore an augmentation quasi-isomorphism
\[
\gamma: \skos_\infty \lto H^d_{\mf{m}}(R)[-d]\,,
\]
which amounts to an epimorphism $R[t_1^{-1},\ldots,t_d^{-1}] \theta_1 \cdots \theta_d \lto H^d_{\mf{m}}(R)$. Given $r \in R$ and integers $e_1,\ldots,e_d \ge 0$ one introduces the \emph{generalised fraction} \cite{Sharp82,Sharp82b}
\[
\begin{bmatrix} r \\ t_1^{e_1},\ldots,t_d^{e_d} \end{bmatrix} := (-1)^d \gamma\left(\frac{r \cdot \theta_1 \cdots \theta_d}{t_1^{e_1}\cdots t_d^{e_d}}\right) \in H^d_{\mf{m}}(R).
\]
We also use the notation $\big[ r \,/\, t_1^{e_1},\ldots,t_d^{e_d} \, \big]$. This can be viewed as a \v{C}ech cocycle, since
\begin{equation}\label{eq:cech_present_kos}
H^d_{\mf{m}}(R) = H\Big( \cdots \lto \oplus_{i_1 < \cdots < i_{d-1}} R[t_{i_1}^{-1},\ldots,t_{i_{d-1}}^{-1}] \lto R[t_1^{-1},\ldots,t_d^{-1}] \Big)\,.
\end{equation}
It is clear from the definition that acting on a generalised fraction $[ r \,/\, t_1^{e_1},\ldots,t_d^{e_d} ]$ with some $t_i$ has the effect of decreasing the exponent of $t_i$ in the denominator and that if $e_i = 1$ then $t_i \cdot [r\,/\, t_1^{e_1},\ldots,t_d^{e_d}] = 0$ in $H^d_{\mf{m}}(R)$. In the same way, we define generalised fractions in $H^d_{\mf{m}}(N)$ for any $R$-module $N$. 

In special cases the description of injective envelopes in terms of inverse polynomials goes back to Gabriel \cite{Gabriel}, Hartshorne \cite{Hartshorne} and Northcott \cite{Northcott}. Our presentation largely follows the one in Lipman's monograph \cite{Lipman84} or Kunz's recent book \cite{Kunz08}. See also \cite[\S 3 -- \S 4]{Lipman05}.

While $\skos_\infty(\bs{t})$ is defined using a specific system of parameters $\bs{t}$, it has a universal property in the derived category and therefore $H^d_{\mf{m}}(R)$ is independent of this choice, up to canonical isomorphism. We will have to manipulate generalised fractions in $H^d_{\mf{m}}(R)$ with denominators given by different regular sequences, and these can be related by the so-called transformation rule or transition formula contained in the next proposition. Note that if $\bs{s},\bs{t}$ are systems of parameters then for some $n > 0$ we have $(s_1^n, \ldots, s_d^n)R \subseteq (t_1,\ldots,t_d)R$.

\begin{proposition}\label{prop:transition_det} Let $\bs{t}, \bs{s}$ be systems of parameters such that $(s_1,\ldots,s_d)R \subseteq (t_1,\ldots,t_d)R$ with $s_i = \sum_j a_{ij} t_j$ $(a_{ij} \in R)$. Then as elements of $H^d_{\mf{m}}(R)$ we have
\[
\begin{bmatrix} 1 \\ t_1,\ldots,t_d \end{bmatrix} = \begin{bmatrix}\det(a_{ij}) \\ s_1,\ldots,s_d \end{bmatrix}.
\]
\end{proposition}
\begin{proof}
See for example \cite[Corollary 2.8]{Lipman87} or \cite[Theorem 4.18]{Kunz08}.
\end{proof}

Given $x \in T^0$ the image under the complete resolution $\rho: T \lto M$ is denoted $\overline{x} \in M$. We are now prepared to state the main theorem.

\begin{theorem}\label{thm:construct_cinjres} Let $\bs{t}$ be a system of parameters with $t_i \cdot 1_T$ null-homotopic for $1 \le i \le d$ and choose for each $i$ a homotopy $\lambda_i$ on $T$ with $\lambda_i \circ \partial + \partial \circ \lambda_i = t_i \cdot 1_T$. Then the morphism
\[
\vartheta: M \lto T \otimes H^d_{\mf{m}}(R)[1-d]
\]
defined by
\begin{equation}\label{eq:construct_cinjres}
\vartheta(\overline{x}) = (-1)^{\binom{d+1}{2}} \lambda_1 \cdots \lambda_d \circ \partial(x) \otimes \begin{bmatrix} 1 \\ t_1,\ldots,t_d \end{bmatrix}
\end{equation}
is a complete injective resolution. Up to homotopy $\vartheta$ is independent of the system of parameters $\bs{t}$ and homotopies $\lambda_i$.
\end{theorem}

The proof will occupy the rest of this section. In outline: since $M$ is the cokernel of the differential $\partial^{-1}$ there is a unique map $i: M \lto T^1$ with $i \circ \rho^0 = \partial^0$, and this defines a morphism of complexes $i: M \lto T[1]$. We will prove that the projection $\varepsilon: T \otimes \skos_\infty \lto T$ is a homotopy equivalence and produce, using the perturbation lemma, an explicit inverse $\iota_\infty$ involving the $\lambda_i$. The composite
\begin{equation}\label{eq:defnvartheta}
\xymatrix{
M \ar[r]^-{i} & T[1] \ar[r]^-{\iota_\infty} & T \otimes \skos_\infty [1] \ar[r]^-{\gamma} & T \otimes H^d_{\mf{m}}(R)[1-d]
}
\end{equation}
turns out to be described by (\ref{eq:construct_cinjres}), and we prove that it has the necessary universal property.

\begin{lemma} There exists a system of parameters $\bs{t}$ such that $t_i \cdot 1_T$ is null-homotopic for $1 \le i \le d$.
\end{lemma}
\begin{proof}
The $R$-module $\cat{T}(T,T)$ of homotopy equivalence classes of self-maps of $T$ has finite-length by Lemma \ref{lemma:homs_finite_length}, so the annihilator contains a system of parameters.
\end{proof}

\begin{lemma} $T \otimes H^d_{\mf{m}}(R)$ is an acyclic complex of injective $R$-modules.
\end{lemma}
\begin{proof}
Since $H^d_{\mf{m}}(R)$ is injective, it suffices to prove that this complex is acyclic. The argument is standard: given an injective $R$-module $J$ and a pair of integers $j < i$ the brutal truncation $T_{\le i}$ is a free resolution of the module of cocycles $Z^{i+1} T$, so we have
\[
H^j( T \otimes J ) = H^j( T_{\le i} \otimes J ) = \textup{Tor}_{i-j}( Z^{i+1} T, J ) 
\]
which vanishes for $j < i - d$ since $J$ has projective dimension $\le d$ (here we use that $R$ is Gorenstein). We conclude that $T \otimes J$ is acyclic.
\end{proof}

From now on we suppose that a system of parameters $\bs{t}$ and sequence of null-homotopies $\{ \lambda_i \}_{i=1}^d$ has been chosen, and we construct the inverse to the projection $\varepsilon: T \otimes \skos_\infty \lto T$. The basic idea is to begin with a homotopy equivalence between $T \otimes \skos_\infty$ and $T$ with the differential $\delta$ on $\skos_\infty$ ``turned off'' and then perturb this differential back in while maintaning the homotopy equivalence.

\begin{definition}\label{defn:2perp} A \emph{deformation retract datum} of complexes of $R$-modules consists of a diagram
\[
\xymatrix@C+2pc{
(L,b) \ar@<-0.8ex>[r]_{\iota} & (M,b), \ar@<-0.8ex>[l]_p
} \quad h
\]
where $(L,b)$ and $(M,b)$ are complexes, $p$ and $\iota$ are morphisms of complexes, and $h$ is a degree one $R$-linear map $M \lto M$, which together satisfy the following two conditions:
\begin{itemize}
\item[(i)] $p\iota = 1$,
\item[(ii)] $\iota p = 1 + bh + hb$.
\end{itemize}
Notice that in particular $p$ is a homotopy equivalence with inverse $\iota$.
\end{definition}

Suppose we are given a deformation retract datum. A degree one $R$-linear map $\mu: M \lto M$ is a \emph{small perturbation} if $(b + \mu)^2 = 0$ and $(\mu h)^n = 0$ for all sufficiently large integers $n$. In this case
\[
(1 - \mu h)\sum_{n \ge 0} (\mu h)^n = 1
\]
so $1-\mu h$ is an isomorphism of graded $R$-modules, and we set
\[
A = (1 - \mu h)^{-1} \mu = \sum_{n \ge 0} (\mu h)^n \mu.
\]
Consider the following collection of data:
\begin{equation}\label{eq:new_perturb_fac}
\xymatrix@C+2pc{
(L,b) \ar@<-0.8ex>[r]_{\iota_\infty} & (M,b + \mu), \ar@<-0.8ex>[l]_{p}
} \quad h_\infty
\end{equation}
where
\begin{equation}\label{eq:new_perturb_fac2}
\iota_\infty = \iota + hA\iota, \quad h_\infty = h + hAh\,.
\end{equation}

The next result is known as the \emph{homological perturbation lemma} \cite{Shih62,Brown67,Gugenheim}.

\begin{theorem}\label{theorem:perp} If $\mu$ is a small perturbation and
\[
p h = 0, \quad p \mu = 0
\]
then (\ref{eq:new_perturb_fac}) is a deformation retract datum.
\end{theorem}
\begin{proof}
See \cite[Theorem 2.3]{Crainic}.
\end{proof}

Given a sequence $\bs{i} = i_1 < \cdots < i_p$ we define $R[t^{-1}_{\bs{i}}] := R[t_{i_1}^{-1},\ldots,t_{i_p}^{-1}]$ and $\theta_{\bs{i}} := \theta_{i_1} \cdots \theta_{i_p}$. The complex $( T \otimes \skos_\infty, \partial \otimes 1 )$ is a direct sum of $T$ and $T[t^{-1}_{\bs{i}}] \theta_{\bs{i}}$ for various sequences $\bs{i}$. In the notation of the theorem $\lambda_i \circ \partial + \partial \circ \lambda_i = t_i \cdot 1_T$ and therefore over $R[t_i^{-1}]$
\[
\big( t_i^{-1} \lambda_i \big) \circ \partial + \partial \circ \big( t_i^{-1} \lambda_i\big) = 1_T\,.
\]
Thus each complex $T[t^{-1}_{\bs{i}}] \theta_{\bs{i}}$ is contractible, with $t_j^{-1} \lambda_j$ giving a contracting homotopy for any $j \in \bs{i}$. For convenience, we use $j = i_1$ in what follows. The upshot is that the inclusion of $T$ as a subcomplex of $( T \otimes \skos_\infty, \partial \otimes 1 )$ is a homotopy equivalence. We can express this in terms of a deformation retract by introducing the $R$-linear homotopy
\[
h: T \otimes \skos_\infty \lto T \otimes \skos_\infty\,, \qquad h = \sum_{\bs{i} = i_1 < \cdots < i_p} h_{\bs{i}}
\]
where for the empty sequence $\bs{i} = \emptyset$ we set $h_{\bs{i}} = 0$, and for sequences of positive length we define
\begin{gather*}
h_{\bs{i}}: T[t^{-1}_{\bs{i}}] \theta_{\bs{i}} \lto T[t^{-1}_{\bs{i}}] \theta_{\bs{i}} \,,\\
h_{\bs{i}} = t_{i_1}^{-1} \lambda_{i_1}\,.
\end{gather*}
It is then easily checked that there is a deformation retract datum
\begin{equation}
\xymatrix@C+2pc{
(T,\partial) \ar@<-0.8ex>[r]_-{\iota} & (T \otimes \skos_\infty,\partial \otimes 1), \ar@<-0.8ex>[l]_-{\varepsilon}
} \quad -h
\end{equation}
where $\iota$ is the inclusion of $T$ into $T \otimes \skos_\infty$. Since $K_\infty$ is bounded it is clear that $\mu = 1 \otimes \delta$ is a small perturbation on $T \otimes \skos_\infty$, so as a consequence of the perturbation lemma we have:

\begin{lemma}\label{lemma:projishe} There is a deformation retract datum
\[
\xymatrix@C+2pc{
(T,\partial) \ar@<-0.8ex>[r]_-{\iota_\infty} & (T \otimes \skos_\infty,\partial \otimes 1 + 1 \otimes \delta), \ar@<-0.8ex>[l]_-{\varepsilon}
} \quad h_\infty
\]
where $\iota_\infty = \sum_{m \ge 0} (-1)^m (h \delta)^m \iota$. In particular, $\varepsilon$ is a homotopy equivalence with inverse $\iota_\infty$.
\end{lemma}

We define $\vartheta$ to be the composite in (\ref{eq:defnvartheta}), noting that there is an sign due to the isomorphism
\begin{equation}\label{eq:commutesignd}
T \otimes \big( H^d_{\mf{m}}(R)[-d] \big) \cong (T \otimes H^d_{\mf{m}}(R))[-d]\,.
\end{equation}

\begin{lemma}\label{lemma:varthetaisok} The map $\vartheta$ is given by the formula in (\ref{eq:construct_cinjres}).
\end{lemma}
\begin{proof}
Let a homogeneous element $x \in T$ be given. Then
\[
\gamma \iota_\infty (x) = \sum_{m \ge 0} (-1)^m \gamma (h \delta)^m \iota (x)
\]
Since $h$ decreases the $T$-degree by one and $\delta$ increases the $\skos_\infty$-degree by one, only the $m = d$ term survives after applying $\gamma$, so $\gamma \iota_\infty(x) = (-1)^d \gamma (h\delta)^d \iota(x)$. Here $\delta = 1 \otimes \delta$ applies to $T \otimes \skos_\infty$ with Koszul signs, as we move the $\theta$'s past elements of $T$.

Since $\delta$ is left multiplication by $\sum_i \theta_i$ the product $(h \delta)^d$ expands as a sum of $d^2$ terms. It follows from (\ref{eq:cech_present_kos}) that in $H^d_{\mf{m}}(R)$ a fraction without a full complement of $t_i$'s vanishes, and since $h$ applied to $x \cdot \theta_{\bs{i}}$ multiplies by $t_{i_1}^{-1}$ the only one of these $d^2$ summands which contributes a nonzero cocycle in local cohomology is the one where we apply $h \theta_d$, then $h \theta_{d-1}$, and so on; that is,
\begin{align*}
\gamma (h \delta)^d \iota(x) &= \gamma h \theta_1 \cdots h \theta_d \iota(x)\\
&= (-1)^{d|x| + \binom{d}{2}} \gamma\left( \frac{\lambda_1 \cdots \lambda_d(x) \cdot \theta_1 \cdots \theta_d }{t_1 \cdots t_d} \right)\,.
\end{align*}
Multiplying by the extra factor of $(-1)^d$ and accounting for the sign from (\ref{eq:commutesignd}), this recovers the formula of (\ref{eq:construct_cinjres}). Note that if $x \in T^0$ then $i(\overline{x}) = \partial(x)$.
\end{proof}

It remains to check that $\vartheta$ has the desired universal property.

\begin{lemma}\label{lemma:iisgood} Let $I$ be an acyclic complex of injective $R$-modules. Then the induced map
\[
\Hom_R(i, I): \Hom_R(T[1], I) \lto \Hom_R(M,I)
\]
is a quasi-isomorphism.
\end{lemma}
\begin{proof}
We will prove that $\Hom_{\K(R)}(T[1], I) \lto \Hom_{\K(R)}(M,I)$ is surjective and leave the proof of injectivity to the reader. If $f: M \lto I$ is a morphism of complexes then there is a factorisation $\overline{f}: M \lto Z$, where $Z = Z^0(I)$, as in the commutative diagram
\[
\xymatrix{
\cdots \ar[r] & T^{-1} \ar[r]^{-\partial} & T^0 \ar[dr]_{-\rho} \ar[rr]^-{-\partial} & & T^1 \ar[r]^{-\partial} & T^2 \ar[r] & \cdots\\
& & & M \ar@/^/[ddr]^f \ar[d]_{\overline{f}} \ar[ur]_i\\
& & & Z \ar[dr]\\
\cdots \ar[r] & I^{-2} \ar[r] & I^{-1} \ar[ur] \ar[rr] & & I^0 \ar[r] & I^1 \ar[r] & \cdots
}
\]
Using projectivity of the top row and acyclicity of the bottom row, we lift $\overline{f}$ to a sequence of maps $F^i: T^{i+1} \lto I^i$ making the diagram commute in degrees $i \le -1$. Then using acyclicity of the top row and injectivity of the bottom row we produce the $F^i$ for $i \ge 0$, and together these maps define a morphism $F: T[1] \lto I$ lifting $f$.
\end{proof}

\begin{lemma}\label{lemma:tskosvs} Let $I$ be an acyclic complex of injective $R$-modules. Then the induced map
\[
\Hom_R(\gamma, I): \Hom_R( T \otimes \skos_\infty, I ) \lto \Hom_R( T \otimes H^d_{\mf{m}}(R)[-d], I )
\]
is a quasi-isomorphism.
\end{lemma}
\begin{proof}
Let $F = T_{\le 0}$ and $K = T_{\ge 1}[1]$ be the brutal truncations. There is a triangle
\begin{equation}
\xymatrix{
T \ar[r] & F \ar[r] & K \ar[r]^{+} &
}
\end{equation}
from which we deduce morphisms of triangles (we write $J = H^d_{\mf{m}}(R)[-d]$ to avoid clutter)
\begin{equation}\label{eq:tskosvs1}
\xymatrix{
T \otimes \skos_\infty \ar[r]\ar[d] & F \otimes \skos_\infty \ar[r]\ar[d] & K \otimes \skos_\infty \ar[d]\ar[r]^-{+} &\\
T \otimes J \ar[r] & F\otimes J \ar[r] & K \otimes J \ar[r]^-{+} &
}
\end{equation}
and (writing $[-,-] = \Hom_R(-,-)$)
\begin{equation}
\xymatrix{
[T \otimes \skos_\infty, I] & [F \otimes \skos_\infty,I] \ar[l] & [K \otimes \skos_\infty, I] \ar[l] & \ar[l]_-{+}\\
[T \otimes J, I] \ar[u]^{\psi'} & [F\otimes J, I]\ar[u]^{\psi}\ar[l] & [K \otimes J,I] \ar[l]\ar[u]^{\psi''}  & \ar[l]_-{+}\,.
}
\end{equation}
To prove that $\psi'$ is a quasi-isomorphism, it suffices to prove that both $\psi$ and $\psi''$ are. Let $C$ denote the cone of the quasi-isomorphism $\gamma: \skos_\infty \lto J$, so that $C$ is a bounded acyclic complex. Then
\[
[ P \otimes C, I ] \cong [ P, [C,I]]\,.
\]
Any morphism from $C$ to $I$ factors through a bounded below complex of injectives and is therefore null-homotopic, that is, $[C,I]$ is acyclic. Hence $[P \otimes C, I]$ is acyclic and $\psi$ is a quasi-isomorphism. 

The cone of $\psi''$ is $[K \otimes C, I]$. We know that $T \otimes J$ is acyclic, so the first two vertical maps in (\ref{eq:tskosvs1}) are quasi-isomorphisms. Hence the third vertical map is a quasi-isomorphism, and $K \otimes C$ is a bounded below acyclic complex. But then $[ K \otimes C, I]$ is acyclic, so $\psi''$ is a quasi-isomorphism and we are done.
\end{proof}

\begin{proof}[Proof of Theorem \ref{thm:construct_cinjres}] For an acyclic complex of injective $R$-modules $I$ we apply $[-,I] = \Hom_R(-,I)$ to (\ref{eq:defnvartheta}) to obtain a chain of quasi-isomorphisms using Lemma \ref{lemma:tskosvs}, Lemma \ref{lemma:projishe} and Lemma \ref{lemma:iisgood}
\[
\xymatrix{
\big[T \otimes H^d_{\mf{m}}(R)[1-d], I\big] \ar[r] & \big[T \otimes \skos_\infty[1], I\big] \ar[r] & \big[T[1], I\big] \ar[r] & \big[M, I\big]\,.
}
\]
This proves that $\vartheta$ is a complete injective resolution. It only remains to argue that, up to homotopy, $\vartheta$ is independent of $\bs{t}$ and the $\lambda_i$. But this is clear, since the definition in (\ref{eq:defnvartheta}) involves no choices.
\end{proof}

We include the following remarks for completeness: they will not be needed in the sequel.

\begin{remark}\label{remark:minimal_comp_free_res} A complex of finite free $R$-modules $X$ is \emph{minimal} if $k \otimes X$ has zero differential. Any isomorphism class in $\cat{T}$ contains a minimal complex \cite[\S 8]{Avramov} which is unique up to isomorphism in the category of complexes. If $M$ has no free summands and we take both $F$ and $Q$ to be minimal resolutions, then the complete resolution of $M$ constructed in (\ref{eq:constructed_cproj_ress}) is minimal.

A complex of injective $R$-modules $I$ is \emph{minimal} if an endomorphism $f: I \lto I$ is a homotopy equivalence if and only if it is an isomorphism of complexes; equivalently, if for every $n \in \mathbb{Z}$ the inclusion $Z^n(I) \lto I^n$ is an injective envelope \cite[Lemma B.1]{Krause05}. Suppose that $T \in \cat{T}$ is minimal and consider the isomorphism
\begin{equation}\label{eq:minimality_complex_inj}
\Hom_R(k, T \otimes H^d_{\mf{m}}(R)) \cong \Hom_R(k, H^d_{\mf{m}}(R)) \otimes T \cong k \otimes T\,.
\end{equation}
Set $I = T \otimes H^d_{\mf{m}}(R)$. Since $\Hom_R(k, I)$ has zero differential, for $n \in \mathbb{Z}$ the socle of $I^n$ is contained in $Z^n(I)$. Every element of $I^n$ is $\mf{m}$-torsion, so it follows that the inclusion $Z^n(I) \lto I^n$ is essential, and thus $I$ is minimal. Hence if $M$ is a CM $R$-module and $T$ a minimal complete free resolution of $M$, Theorem \ref{thm:construct_cinjres} will produce a \emph{minimal} complete injective resolution.
\end{remark}

\begin{remark}\label{remark:orderingandsigns} The homotopy class of $\vartheta$ is independent of the ordering of the symbols $\lambda_1,\ldots,\lambda_d,\partial$ in (\ref{eq:construct_cinjres}), up to signs: applying these maps in any order defines a morphism $M \lto T \otimes H^d_{\mf{m}}(R)[1-d]$ homotopic to $\textup{sgn}(\sigma)\vartheta$ where $\sigma$ is the corresponding permutation on $d+1$ letters. 

To see this note that $\lambda_i\circ  \partial + \partial \circ \lambda_i = t_i \cdot 1_T$ annihilates any generalised fraction with $t_1,\ldots,t_d$ in the denominator, so $\partial$ effectively anticommutes with $\lambda_i$ in (\ref{eq:construct_cinjres}). Let $\bs{t}'$ be $\bs{t}$ with $t_j$ and $t_{j+1}$ interchanged. By the transformation rule (Proposition \ref{prop:transition_det}) there is an equality $[ 1 \,/\, \bs{t} ] = - [1\,/\, \bs{t}']$ of generalised fractions, and hence $\lambda_j$ and $\lambda_{j+1}$ anticommute in (\ref{eq:construct_cinjres}).
\end{remark}

\begin{remark} Complete injective resolutions give an embedding of the singularity category into a compactly generated triangulated category, for any separated noetherian scheme $X$. Using Brown representability Krause proves in \cite{Krause05} that complete injective resolutions exist in this generality, and that taking complete injective resolutions defines a fully faithful functor
\[
\ci(-): \D_{\sg}(X) \lto \K_{\ac}(\Inj X)
\]
where $\K_{\ac}(\Inj X)$ is the homotopy category of acyclic complexes of injective quasi-coherent sheaves. Moreover, he shows that this category is compactly generated and that the image of the embedding $\ci(-)$ is, up to split idempotents, exactly the subcategory of compact objects.
\end{remark}

\section{Computing the pairing}\label{section:duality}

Let $(R,\mf{m},k)$ be a local Gorenstein ring of Krull dimension $d$ with an isolated singularity and define
\begin{equation}
\cat{T} = \K_{\ac}(\free R)\,.
\end{equation}
Given $X,Y \in \cat{T}$ the isomorphism of Theorem \ref{theorem:ausdual} corresponds to a nondegenerate pairing
\begin{equation}\label{eq:duality1}
\plangle -, - \prangle: \cat{T}(Y, X[d-1]) \otimes \cat{T}(X,Y) \lto H^d_{\mf{m}}(R)\,.
\end{equation}
In this section we give an explicit formula for this pairing, using the construction of complete~injective resolutions in Section \ref{section:cinjres}. When $R$ is a $k$-algebra this can be refined to a nondegenerate pairing taking values in $k$, by taking residues of generalised fractions. Our convention is that functions taking values in local cohomology (resp. in $k$) are given double brackets $\plangle - \prangle$ (resp. ordinary brackets $\langle - \rangle$).

We will see that the pairing $\plangle -, - \prangle$ factors as composition followed by a ``pretrace'' map
\[
\xymatrix@C+2pc{
\cat{T}(Y, X[d-1]) \otimes \cat{T}(X,Y) \ar[r]^-{- \circ -} & \cat{T}(X,X[d-1]) \ar[r]^-{\plangle - \prangle} & H^d_{\mf{m}}(R)\,,
}
\]
that is $\plangle \psi, \phi \prangle = \plangle \psi \circ \phi \prangle$. We are going to first describe a formula for $\plangle - \prangle$, and then prove that~the pairing determined by this formula is indeed the one produced by Theorem \ref{theorem:ausdual}. To this end we~fix~a complex $X \in \cat{T}$ with differential $\partial$ and, as in the previous section, we choose a system of parameters~$\bs{t}$ acting null-homotopically on $X$ together with null-homotopies $\lambda_i$ on $X$ such that
\begin{equation}\label{eq:homotopylambdasec4}
\lambda_i \circ \partial + \partial \circ \lambda_i = t_i \cdot 1_X\,.
\end{equation}
Given a morphism $\alpha: X \lto X[d-1]$ we define the degree zero operator $\holn_\alpha$ on $X$ as in Equation (\ref{eq:defnptr0}) of the introduction. Recall the claim made in the introduction that the class
\begin{equation}\label{eq:defnptragain}
\plangle \alpha \prangle := (-1)^{\binom{d+1}{2}}\begin{bmatrix} \tr( \holn^i_\alpha ) \\ t_1, \ldots, t_d \end{bmatrix} = (-1)^{\binom{d+1}{2} + i} \begin{bmatrix} \tr\left( \alpha \circ \lambda_1 \cdots \lambda_d \circ \partial \right)^i \\ t_1, \ldots, t_d \end{bmatrix} \in H^d_{\mf{m}}(R)
\end{equation}
is independent of all choices: the integer $i$, the system of parameters $\bs{t}$, and the null-homotopies $\lambda_j$. Independence of the second two choices will follow from the next theorem, and independence of $i$ is the statement of the following lemma. We write $\plangle - \prangle_X$ for $\plangle - \prangle$ if we want to emphasise $X$.

\begin{lemma}\label{lemma:ptrandinteger} For $i \in \mathbb{Z}$ we have the following equality in $H^d_{\mf{m}}(R)$:
\[
\begin{bmatrix} \tr( \holn^i_\alpha ) \\ t_1, \ldots, t_d \end{bmatrix} = \begin{bmatrix} \tr( \holn^0_\alpha ) \\ t_1, \ldots, t_d \end{bmatrix}\,.
\]
\end{lemma}
\begin{proof}
It suffices to prove that $\tr( \holn^i_\alpha ) = \tr( \holn^0_\alpha )$ in $R/(\bs{t})$. But using (\ref{eq:homotopylambdasec4}) we have
\begin{align*}
\tr( \holn^0_\alpha ) &= \tr\left( \alpha^{1-d} \circ \lambda_1^{2-d} \cdots \lambda_d^1 \circ \partial^0 \right)\\
&= (-1)^d \tr\left( \alpha^{1-d} \circ \partial^{-d} \circ \lambda_1^{1-d} \cdots \lambda_d^0 \right)\\
&= - \tr\left( \partial^{-1} \circ \alpha^{-d} \circ \lambda_1^{1-d} \cdots \lambda_d^0 \right)\\
&= - \tr\left( \alpha^{-d} \circ \lambda_1^{1-d} \cdots \lambda_d^0 \circ \partial^{-1} \right)\\
&= - \tr( \holn^{-1}_\alpha )\,.
\end{align*}
Continuing to ``rotate'' $\partial$ through the trace in this direction takes care of all $i < 0$, and for $i > 0$ we simply rotate the other way.
\end{proof}

The main theorem states that the functional $\plangle - \prangle$ determines Auslander's duality.

\begin{theorem}\label{theorem:genfracformula} The nondegenerate pairing of (\ref{eq:duality1}) takes the value
\[
\plangle \psi, \phi \prangle = \plangle \psi \circ \phi \prangle = (-1)^{\binom{d+1}{2}}\begin{bmatrix} \tr\left( \psi \circ \phi \circ \lambda_1 \cdots \lambda_d \circ \partial \right)^0 \\ t_1, \ldots, t_d \end{bmatrix}
\]
for a pair of morphisms $\psi: Y \lto X[d-1]$ and $\phi: X \lto Y$.
\end{theorem}
\begin{proof}
If we set $M = \Coker(\partial^{-1})$ then the morphism $\vartheta: M \lto X \otimes H^d_{\mf{m}}(R)[1-d]$ defined by
\begin{gather*}
\vartheta(\overline{x}) = (-1)^{\binom{d+1}{2}} \lambda_1 \cdots \lambda_d \circ \partial(x) \otimes \big[ 1 \; / \; t_1, \ldots, t_d \,\big]
\end{gather*}
is a complete injective resolution, by Theorem \ref{thm:construct_cinjres}. If we begin with a morphism $\psi: Y \lto X[d-1]$ at the end of the chain of quasi-isomorphisms in the proof of Theorem \ref{theorem:ausdual}, then the corresponding function $\Hom_R(X,Y) \lto H^d_{\mf{m}}(R)$ evaluates on $\phi: X \lto Y$ to the image of the identity under the composite
\[
\xymatrix@C+3pc{
X^{\dual} \otimes M \ar[r]^-{1 \otimes \vartheta} & X^{\dual} \otimes X \otimes H^d_{\mf{m}}(R)[1-d] \ar[r]^-{1 \otimes (\psi \circ \phi) \otimes 1} & X^{\dual} \otimes X \otimes H^d_{\mf{m}}(R) \ar[d]^{\eval \otimes 1}\\
R \ar[u]^{\textup{coev}} & & H^d_{\mf{m}}(R)\,.
}
\]
Here $\textup{coev}: R \lto X^{\dual} \otimes M$ denotes the map sending $1 \in R$ to $\sum_i e_i^* \otimes \overline{e_i}$ for a basis $e_i$ of $X^0$, and $\eval: X^{\dual} \otimes X \lto R$ is the usual evaluation map. For readability we omit the canonical isomorphisms used to pull shifts out of components of the tensor product. We conclude that
\begin{align*}
\plangle \psi, \phi \prangle &= \sum_i \eval ( 1 \otimes \psi \circ \phi \otimes 1 )(1 \otimes \vartheta)(e_i^* \otimes \overline{e_i})\\
&= (-1)^{\binom{d+1}{2}} \sum_i \eval\left( e_i^* \otimes \psi \circ \phi \circ \lambda_1 \cdots \lambda_d \circ \partial(e_i) \right) \cdot \big[ 1 \; / \; t_1, \ldots, t_d \,\big]\\
&= (-1)^{\binom{d+1}{2}} \tr\left( \psi \circ \phi \circ \lambda_1 \cdots \lambda_d \circ \partial \right)^0 \cdot \big[ 1 \; / \; t_1, \ldots, t_d \,\big]\\
&= \plangle \psi \circ \phi \prangle
\end{align*}
as claimed.
\end{proof}
Finally we enumerate some basic properties of the pretrace map
\[
\plangle - \prangle: \cat{T}(X,X[d-1]) \lto H^d_{\mf{m}}(R)\,.
\]
Since the pairing $\plangle -, - \prangle$ is canonically defined and $\plangle \psi \prangle = \plangle \psi, 1 \prangle$ we see that $\plangle - \prangle$ does not depend on a choice of system of parameters or null-homotopies.

\begin{lemma}\label{lemma:ptrproperties} The pretrace map has the following properties:
\begin{itemize}
\item[(i)] For morphisms $\psi: Y \lto X[d-1]$ and $\phi: X \lto Y$, $\plangle \psi \circ \phi \prangle = \plangle \phi \circ \psi \prangle$.
\item[(ii)] For a morphism $\psi: X \lto X[d-1]$, $\plangle \psi \prangle_X = (-1)^d \cdot \plangle \psi \prangle_{X[1]}$.
\end{itemize}
\end{lemma}
\begin{proof}
(i) is an immediate consequence of the naturality of the isomorphism in Theorem \ref{theorem:ausdual}. If we use the null-homotopies $\lambda'_j = -\lambda_j$ for the action of $t_j$ on $X[1]$, then (ii) follows from Lemma \ref{lemma:ptrandinteger}.
\end{proof}

\begin{remark}\label{remark:gradedcomm} In the formula (\ref{eq:defnptragain}) the maps $\psi, \lambda_1,\ldots,\lambda_d, \partial$ are graded commutative, that is, if $a,b$ stand for one of these maps then interchanging $ab$ with $(-1)^{|a||b|} ba$ does not change the cohomology class of the fraction in $H^d_{\mf{m}}(R)$. This follows by the arguments of Remark \ref{remark:orderingandsigns}, and Lemma \ref{lemma:ptrandinteger}.
\end{remark}

\subsection{The case of $k$-algebras}\label{section:dualisingpairsyo}

Let us now assume in addition that $R$ is a $k$-algebra. In this section we elaborate on the isomorphism (\ref{eq:firstoccurdualisingpair}) and its consequences for duality. An $R$-module $M$ is called \emph{$\mf{m}$-torsion} if every element of $M$ is annihilated by some power of $\mf{m}$. The Hom-spaces $\cat{T}(X,Y)$ are finite-length and therefore $\mf{m}$-torsion, and it is evident from (\ref{eq:defn_localcohomology}) that the infinite module $H^d_{\mf{m}}(R)$ is also $\mf{m}$-torsion. 

A \emph{dualising pair} $(E,\zeta)$ is an $\mf{m}$-torsion module $E$ together with a $k$-linear map $\zeta: E \lto k$ which is universal, in the sense that for every $\mf{m}$-torsion module $M$ the map induced by $\zeta$
\begin{equation}\label{eq:dualising_pair_defn}
\Hom_R(M,E) \lto \Hom_k(M,k)
\end{equation}
is an isomorphism. Clearly dualising pairs are unique up to isomorphism. 

Every $\mf{m}$-torsion module is a direct limit of its finite-length submodules and $\Hom_k(-,k)$ is exact, so (\ref{eq:dualising_pair_defn}) is an isomorphism for all $M$ if and only if it is an isomorphism for $M = k$. Thus $(E, \zeta)$ is a dualising pair if and only if the socle $\Hom_R(k,E)$ is one-dimensional and $\zeta$ is nonzero on the socle. In this case there is a unique $R$-linear map $\iota: k \lto E$ such that $\zeta \circ \iota = 1$, and one can check that $E$ is an injective $R$-module and that $\iota$ is an injective envelope. 

Reversing this, it is easy to produce a dualising pair abstractly: if $E$ is an injective envelope of $k$ then any $k$-linear map $\zeta: E \lto k$ which does not vanish on the socle must be a dualising pair; see for example \cite[Proposition 0.4]{Salas02}. Since $R$ is Gorenstein we know that $H^d_{\mf{m}}(R)$ is an injective envelope of $k$, so a dualising pair $(H^d_{\mf{m}}(R), \zeta)$ exists.

The following is immediate from Theorem \ref{theorem:ausdual} once we choose a dualising pair:

\begin{corollary} Given $X, Y \in \cat{T}$ there is a nondegenerate pairing
\[
\langle -, - \rangle: \cat{T}(Y, X[d-1]) \otimes_k \cat{T}(X,Y) \lto k
\]
defined by $\langle \psi, \phi \rangle = \zeta \plangle \psi, \phi \prangle$ which is natural in both variables and compatible with suspension.
\end{corollary}

\begin{remark} Let $S$ denote the functor $(-)[d-1]$ on $\cat{T}$. Together with the isomorphism $S \circ [1]  \cong [1] \circ S$ given by $(-1)^{d-1} \cdot 1_{[d]}$ this is a triangulated functor, and it follows from Lemma \ref{lemma:ptrproperties} and the previous corollary that $\cat{T}$ is $(d-1)$-Calabi-Yau in the sense of \cite[Proposition 2.2]{Keller08}.
\end{remark}

To obtain a concrete nondegenerate pairing on the morphism spaces of $\cat{T}$ it only remains to find an explicit $k$-linear functional $\zeta$ on $H^d_{\mf{m}}(R)$ which is nonvanishing on the socle. Such a functional is unique up to an automorphism, but fixing this automorphism is quite subtle: while any element of $H^d_{\mf{m}}(R)$ can be presented as a generalised fraction, this presentation is not unique, so it is a challenge to assign scalars to fractions in a way which is well-defined. 

This is the classical problem of defining residues of meromorphic differential forms on algebraic varieties, solved by Serre \cite[Ch. II]{Serre59} and Tate \cite{Tate68} for curves and generalised by Grothendieck to arbitrary varieties \cite{ResiduesDuality,Conrad00}. The theory of residue symbols is extensive, and we only sketch the parts we need; for more details see \cite[\S 5.3]{Lipman01}, \cite[pp.64--67]{Lipman84} or \cite{Lipman87,Kunz90,Kunz08}.

We begin with residues over power series rings, and then move on to singular rings. 

\begin{example}\label{example:residuespowerseries} If $S = k\llbracket x_1,\ldots,x_n \rrbracket$ there is a canonical $k$-linear map $\Res_{S/k}: H^n_{\mf{m}}(\omega_S) \lto k$ with
\begin{equation}\label{eq:eval_rule_gen_frac}
\Res_{S/k}\!\begin{bmatrix} \ud x_1 \wedge \cdots \wedge \ud x_n \\ x_1^{e_1},\ldots,x_n^{e_n} \end{bmatrix} = \begin{cases} 1 & e_1 = \cdots = e_n = 1,\\ 0 & \text{otherwise}. \end{cases}
\end{equation}
Here $\omega_S$ is a suitably defined module of top-degree differential forms. This determines the value of $\Res_{S/k}$ on any generalised fraction with powers of the variables in the denominator. If $\bs{f}$ is a regular sequence of length $n$ in $S$ then there exist integers $e_1,\ldots,e_n \ge 1$ such that $x_i^{e_i} \in (\bs{f})S$, say
\[
x_i^{e_i} = \sum_j a_{ij} f_j.
\]
Then by the transformation rule of Proposition \ref{prop:transition_det} (with $\ud V := \ud x_1 \wedge \cdots \wedge \ud x_n$)
\begin{equation}\label{eq:grothendieck_res_symbol}
\Res_{S/k}\!\begin{bmatrix} s \cdot \ud V \\ f_1,\ldots,f_n \end{bmatrix} = \Res_{S/k}\!\begin{bmatrix} s \cdot \det(a_{ij}) \cdot \ud V \\ x_1^{e_1},\ldots,x_n^{e_n} \end{bmatrix}\,
\end{equation}
so (\ref{eq:eval_rule_gen_frac}) uniquely determines the value of $\Res_{S/k}$ on every generalised fractions. This is called, for obvious reasons, the \emph{residue symbol}. In the case $k = \mathbb{C}$ the residue symbol can also be constructed by integration, see \cite[Chapter V]{Griffiths}. To see that the definition in terms of local cohomology agrees with the analytic definition, it suffices to observe that both constructions obey the transformation rule and both have the same values on the basic fractions in (\ref{eq:eval_rule_gen_frac}).

Here we trivialise $\omega_S \cong S$ via the generator $\ud V$ and note that the map $\Res_{S/k}: H^n_{\mf{m}}(S) \lto k$ thus defined is nonvashing on the socle (the fraction in (\ref{eq:eval_rule_gen_frac}) with all $e_i = 1$ generates the socle, which is one-dimensional). Thus $(H^n_{\mf{m}}(S), \Res_{S/k})$ is a dualising pair for $S$.
\end{example}

From now on suppose that that our local Gorenstein $k$-algebra $R$ is a complete intersection, i.e.
\[
S = k\llbracket x_1, \ldots, x_n \rrbracket, \qquad R = S / (f_1,\ldots,f_c)
\]
for a regular sequence $\bs{f} = (f_1,\ldots,f_c)$ in $\mf{m}$. Hence $n = d + c$. In order to canonically define residues over $R$, one has to introduce modules of regular differential forms \cite{Kunz90}. But if one simply wants to construct a dualising pair for $R$ which is computable in terms of residues over $S$, one can proceed more directly by relating generalised fractions over $R$ and $S$.

\begin{proposition}\label{prop:tauinclusion} There is a well-defined $S$-linear map
\[
\tau: H^d_{\mf{m}}(R) \lto H^n_{\mf{m}}(S),
\]
which for a system of parameters $\bs{t}$ in $R$ is given by
\[
\tau \begin{bmatrix} r \\ t_1, \ldots, t_d \end{bmatrix} = \begin{bmatrix} r \\ f_1, \ldots, f_c, t_1, \ldots, t_d \end{bmatrix}\,.
\]
Moreover $\tau$ is injective and defines an isomorphism between $H^d_{\mf{m}}(R)$ and the submodule of elements of $H^n_{\mf{m}}(S)$ annihilated by the ideal $(\bs{f})$.
\end{proposition}
\begin{proof}
The proof is technical and we have relegated it to Appendix \ref{section:resid_genfrac}. See Lemma \ref{lemma:genfracquotientsrelate}.
\end{proof}

The residue map $\Res_{S/k}: H^n_{\mf{m}}(S) \lto k$ is a dualising pair for $S$, and the composite
\begin{gather*}
\zeta := \Res_{S/k}\, \circ\, \tau: H^d_{\mf{m}}(R) \lto H^n_{\mf{m}}(S) \lto k,\\
\zeta\! \begin{bmatrix} r \\ t_1, \ldots, t_d \end{bmatrix} = \Res_{S/k} \! \begin{bmatrix} r \\ f_1, \ldots, f_c, t_1, \ldots, t_d \end{bmatrix}
\end{gather*}
is easily seen to be a dualising pair for $R$. Combining this dualising pair with Theorem \ref{theorem:genfracformula} we have a more explicit form of the duality for complete intersections:

\begin{corollary}\label{corollary:ciresformula} When $R$ is a complete intersection there is a nondegenerate pairing
\[
\langle -, - \rangle: \cat{T}(Y, X[d-1]) \otimes_k \cat{T}(X,Y) \lto k
\]
defined in the above notation by
\[
\langle \psi, \phi \rangle = (-1)^{\binom{d+1}{2}} \Res_{S/k}\! \begin{bmatrix} \tr_R\left( \psi \circ \phi \circ \lambda_1 \cdots \lambda_d \circ \partial \right)^0 \\ f_1, \ldots, f_c, t_1, \ldots, t_d \end{bmatrix}
\]
which is natural in both variables and compatible with suspension.
\end{corollary}

\section{Matrix factorisations}\label{section:mfs}

Throughout $k$ is a field of characteristic zero, $S = k\llbracket x_1,\ldots,x_n \rrbracket$, and $W \in S$ is a polynomial chosen such that the zero locus $\{ W = 0 \}$ in $\mathbb{A}^n_k$ has an isolated singularity at the origin. Then $R = S/(W)$ is a local Gorenstein $k$-algebra with an isolated singularity, and the results of the previous section apply to the triangulated category $\K_{\ac}(\free R)$. The arguments work more generally for any regular local $k$-algebra $S$, but for simplicity we stick to power series.

Throughout all matrix factorisations will defined over $S$ and factorise $W$, and $\partial_i$ denotes $\partial/\partial x_i$. Recall the construction, given in the introduction, of a $\mathbb{Z}$-graded complex $\perd{X}$ of finite free $R$-modules from any matrix factorisation $X$.

\begin{lemma}[(\cite{Buchweitz, Eisenbud80})] $\perd{X}$ is an acyclic complex of finite free $R$-modules and the functor
\[
\perd{(-)}: \hmf(S, W) \lto \K_{\ac}(\free R)
\]
is an equivalence.
\end{lemma}
\begin{proof}
Let us recall the proof that $\overline{X}$ is acyclic. If $x \in Z^i(\overline{X})$ then $d_X^i(x) \in W \cdot X^{i+1}$ and hence $d^i_X(x) = d^i_X \circ d^{i-1}_X(y)$ for some $y \in X^{i-1}$. But since $d^i_X \circ d^{i-1}_X = W \cdot 1_X$ is injective $d^i_X$ must be injective, whence $x = d^{i-1}_X(y)$ is a coboundary, and $\overline{X}$ is acyclic. There is a diagram of functors
\[
\xymatrix{
\hmf(S,W) \ar[dr]_-{F}\ar[rr]^-{\perd{(-)}} & & \K_{\ac}(\free R)\ar[dl]^-{G}\\
& \underline{\CM}(R)
}
\]
which commutes up to natural isomorphism, where $\underline{\operatorname{CM}}(R)$ is the stable category of maximal Cohen-Macaulay $R$-modules, $G(Y) = \Coker(\partial^{-1}_Y)$ and $F(X) = \Coker(d^{1}_X)$. The functor $G$ is an equivalence by \cite[(4.4.1)]{Buchweitz} and $F$ is an equivalence by \cite{Eisenbud80}, so periodification is also an equivalence.
\end{proof}

We know an explicit formula (see Corollary \ref{corollary:ciresformula}) for the nondegenerate pairing on the morphism spaces of $\cat{T} = \K_{\ac}(\free R)$, since $R$ is a complete intersection with $c = 1$. Using the equivalence of $\cat{T}$ with the homotopy category of matrix factorisations $\m = \hmf(S,W)$ we obtain a nondegenerate pairing on the morphism spaces of $\m$. Here is what we find:

\begin{theorem}\label{theorem:klformula} For matrix factorisations $X,Y$ there is a nondegenerate pairing
\begin{equation}\label{eq:klformulaproof1}
\langle -, - \rangle: \m(Y, X[n]) \otimes_k \m(X,Y) \lto k
\end{equation}
defined by
\[
\langle \psi, \phi \rangle = \frac{1}{n!}(-1)^{\binom{n-1}{2}} \Res_{S/k}\!\begin{bmatrix} \str_S\!\left(\psi \circ \phi \circ \ud_{S/k}(d_X)^{\wedge n}\right) \\ \partial_1 W, \ldots, \partial_n W \end{bmatrix}\,.
\]
which is natural in both variables and compatible with suspension.
\end{theorem}

Regarding the notation: we choose homogeneous bases for $X,Y$ so that the matrix-valued $n$-form
\[
\ud_{S/k}(d_X)^{\wedge n} = \sum_{\sigma \in S_n} \textup{sgn}(\sigma) \cdot \partial_{\sigma(1)}(d_X) \cdots \partial_{\sigma(n)}(d_X) \cdot \ud x_1 \wedge \cdots \wedge \ud x_n
\]
makes sense\footnote{The careful reader will note that we are confusing K\"ahler differentials over the polynomial and power series ring, but this is harmless since we work within a generalised fraction killed by all sufficiently high powers of the variables.}, where $S_n$ is the symmetric group. The \emph{supertrace} of a homogeneous endomorphism $\alpha$ of a finite free $\Ztwo$-graded $S$-module is by definition the trace of $(-1)^F \alpha$. We also note that $(-)[2]$ is the identity functor on $\cat{M}$, so $X[d-1] = X[n-2] = X[n]$.

The proof will be preceeded by a series of lemmas, establishing some basic commutative algebra which we will need. The main mismatch between our current setting and that of Section \ref{section:duality} which needs to be accounted for is the following: to apply our earlier results, we need a system of parameters $\bs{t} = (t_1,\ldots,t_{n-1})$ in $R$ acting null-homotopically on the acyclic complex $\perd{X}$ associated to a matrix factorisation $X$, and associated null-homotopies $\lambda_j$.

This data \emph{almost} comes for free: since the singularity of the hypersurface $\{W = 0\}$ at the origin is isolated the partial derivatives $\bs{w} = (\partial_1 W, \ldots, \partial_n W)$ form a regular sequence\footnote{For $k = \mathbb{C}$ see \cite[Lemma 23]{Greuel07}. The generalisation to arbitrary $k$ is routine.}, and these derivatives act null-homotopically on any matrix factorisation and therefore also on the periodification:

\begin{lemma}\label{lemma:periodofdoys} Let $X$ be a matrix factorisation. The periodification of the map $\partial_i(d_X)$ is a homotopy
\[
\perd{\partial_i(d_X)}: \perd{X} \lto \perd{X}
\]
satisfying $\perd{\partial_i(d_X)} \circ d_{\perd{X}} + d_{\perd{X}} \circ \perd{\partial_i(d_X)} = \partial_i W \cdot 1_{\perd{X}}$.
\end{lemma}
\begin{proof}
Since $(d_X)^2 = W$, this follows by the Leibniz rule $\partial_i(d_X) \circ d_X + d_X \circ \partial_i(d_X) = \partial_i W \cdot 1_X$.
\end{proof}

If we could find a system of parameters for $R$ as a subset of $\bs{w}$, say $(\partial_1 W, \ldots, \partial_{n-1} W)$, then we could take as our null-homotopies the $\overline{\partial_i(d_X)}$ and Corollary \ref{corollary:ciresformula} would provide a residue formula for the pairing in $\m$. This is not quite the Kapustin-Li pairing, but it is very close. 

But we are getting ahead of ourselves: it is \emph{not} true in general that a subset of $\bs{w}$ gives a system of parameters for $R$. However this can always be arranged by a change of variables, or what amounts to the same thing, replacing $\bs{w}$ by the sequence
\begin{equation}\label{eq:defnwprime}
\bs{w}' = \left(\sum_{j=1}^n c_{1j} \cdot \partial_j W, \ldots, \sum_{j=1}^n c_{nj} \cdot \partial_j W\right)
\end{equation}
for some invertible $n \times n$ matrix $C$ over $k$. We can always choose $C$ such that
\begin{equation}\label{eq:tseq}
\bs{t} := (w'_1, \ldots, w'_{n-1})
\end{equation}
is a system of parameters for $R$. This is the content of the next pair of lemmas.

\begin{lemma}\label{lemma:technical_sop} Let $(A,\mf{m})$ be a Cohen-Macaulay local ring of dimension $d > 0$ which is an algebra over an infinite field $K$, and let $u_1, \ldots, u_m$ be elements of $A$ generating an $\mf{m}$-primary ideal. There are linear combinations $y_i = \sum_{j=1}^m a_{ij} u_j$ with coefficients $a_{ij} \in K$ such that $y_1, \ldots, y_d$ forms a system of parameters for $A$.
\end{lemma}
\begin{proof}
Let $\mf{p}_1, \ldots, \mf{p}_r$ be the associated primes of $A$. We claim that there exists a linear combination $y = b_1 u_1 + \cdots + b_m u_m$ $(b_i \in K)$ which is a regular element of $A$. Suppose to the contrary that every such linear combination is a zero-divisor, and therefore belongs to the union $\bigcup_i \mf{p}_i$. If we set $V_i = \{ (b_1,\ldots,b_m) \in K^m \l \sum_j b_j u_j \in \mf{p}_i \}$ for $1 \le i \le r$ then our assumption implies that $\bigcup_i V_i = K^m$. But every $V_i$ is a proper subspace, because if $V_i = K^m$ then $\{ u_1, \ldots, u_m \} \subseteq \mf{p}_i$ which would imply $\mf{p}_i = \mf{m}$, contradicting our assumption that $d > 0$. We have reached the desired contradiction, because $K$ is infinite and thus $K^m$ is not a finite union of proper subspaces. Applying the claim recursively we produce the desired system of parameters for $A$.
\end{proof}

\begin{lemma}\label{lemma:special_reg_seq_exists} There is an invertible matrix $C$ over $k$ such that $(\ref{eq:tseq})$ is a system of parameters for $R$.
\end{lemma} 
\begin{proof}
By hypothesis $\bs{w}$ generates an ideal primary for the maximal ideal in $S$, and therefore also in $R$. Using Lemma \ref{lemma:technical_sop}, we can find $n-1$ vectors $\{ (c_{i1}, \ldots, c_{in}) \in k^n \}_{1 \le i \le n-1}$ such that $\bs{t}$ is a system of parameters for $R$. These vectors must be linearly independent (otherwise we would have $\dim(R) < n - 1$) and we can define the desired matrix $C$ by appending to this list an arbitrary, linearly independent, vector $(c_{n1},\ldots,c_{nn})$ from $k^n$.
\end{proof}

In what follows we assume that $C$ and thus $\bs{w}'$ and $\bs{t}$ have been fixed, such that $\bs{t}$ is a system of parameters for $R$. It follows from Lemma \ref{lemma:periodofdoys} that the homotopies
\[
\lambda_i := \sum_{j=1}^n c_{ij} \cdot \perd{\partial_j(d_X)}
\]
on $\perd{X}$ satisfy $\lambda_i \circ d_{\perd{X}} + d_{\perd{X}} \circ \lambda_i = w'_i \cdot 1_{\perd{X}}$ for $1 \le  i \le n$.

\begin{proof}[Proof of Theorem \ref{theorem:klformula}]
By Corollary \ref{corollary:ciresformula} there is a nondegenerate pairing (\ref{eq:klformulaproof1}) given by
\[
\langle \psi, \phi \rangle = (-1)^{\binom{n-1}{2}} \Res_{S/k}\! \begin{bmatrix} \tr_R\left( \perd{\psi} \circ \perd{\phi} \circ \lambda_1 \cdots \lambda_{n-1} \circ d_{\perd{X}} \right)^0 \\ t_1, \ldots, t_{n-1}, W \end{bmatrix}\,.
\]
Note that the sign $\epsilon = (-1)^{\binom{n-1}{2}}$ above has a contribution from moving $W$ from the left end of the denominator to the right. Because of the extra $\Ztwo$-symmetry of $\perd{X}$, it is convenient to rewrite this in terms of supertraces of $S$-linear endomorphisms of $X$, making use of Lemma \ref{lemma:ptrandinteger}
\[
\langle \psi, \phi \rangle = \frac{1}{2} \epsilon \Res_{S/k}\! \begin{bmatrix} \str_S\left( \psi \circ \phi \circ \lambda_1 \cdots \lambda_{n-1} \circ d_X \right) \\ t_1, \ldots, t_{n-1}, W \end{bmatrix}\,.
\]
Here we abuse notation and write $\lambda_i$ for the map $\sum_j c_{ij} \cdot \partial_j(d_X)$ on $X$. Next we observe that the sequence $(t_1,\ldots,t_{n-1}, W \cdot w'_n) = (w'_1, \ldots, w'_{n-1}, W \cdot w'_n)$ is regular, and by the transformation rule
\begin{align*}
\langle \psi, \phi \rangle &= \frac{1}{2}\epsilon \Res_{S/k}\! \begin{bmatrix} \str_S\left( \psi \circ \phi \circ \lambda_1 \cdots \lambda_{n-1} \circ ( w'_n \cdot 1_X ) \circ d_X \right) \\ w'_1, \ldots, w'_{n-1}, W \cdot w'_n \end{bmatrix}\\
&= \frac{1}{2}\epsilon \Res_{S/k}\! \begin{bmatrix} \str_S\left( \psi \circ \phi \circ \lambda_1 \cdots \lambda_{n-1} \circ ( \lambda_n \circ d_X + d_X \circ \lambda_n ) \circ d_X \right) \\ w'_1, \ldots, w'_{n-1}, W \cdot w'_n \end{bmatrix}\\
&= \frac{1}{2}\epsilon \Res_{S/k}\! \begin{bmatrix} \str_S\left( \psi \circ \phi \circ \lambda_1 \cdots \lambda_{n-1} \circ \lambda_n \circ (d_X)^2 \right) + \delta \\ w'_1, \ldots, w'_{n-1}, W \cdot w'_n \end{bmatrix}\,.
\end{align*}
where $\delta = \str_S( \psi \circ \phi \circ \lambda_1 \cdots \lambda_{n-1} \circ d_X \circ \lambda_n \circ d_X )$. The $(d_X)^2 = W \cdot 1_X$ cancels with the $W$ in the denominator, so that
\begin{equation}\label{eq:almostfinalklres}
\langle \psi, \phi \rangle = \frac{1}{2}\epsilon \Res_{S/k}\! \begin{bmatrix} \str_S\left( \psi \circ \phi \circ \lambda_1 \cdots \lambda_n \right) \\ w'_1, \ldots, w'_n \end{bmatrix} + \frac{1}{2}\epsilon \Res_{S/k}\! \begin{bmatrix} \delta \\ w'_1, \ldots, w'_{n-1}, W \cdot w'_n \end{bmatrix}\,.
\end{equation}
 Working modulo the sequence $w'_1, \ldots, w'_{n-1}, W \cdot w'_n$, so that $d_X$ anticommutes with every $\lambda_j$ except for $\lambda_n$, we have
\begin{align*}
\delta &= \str_S( \psi \circ \phi \circ \lambda_1 \cdots \lambda_{n-1} \circ d_X \circ \lambda_n \circ d_X )\\
&= - \str_S( d_X \circ \psi \circ \phi \circ \lambda_1 \cdots \lambda_{n-1} \circ d_X \circ \lambda_n )\\
&= (-1)^{n+1} \str_S( \psi \circ \phi \circ d_X \circ \lambda_1 \cdots \lambda_{n-1} \circ d_X  \circ \lambda_n)\\
&= \str_S( \psi \circ \phi \circ \lambda_1 \cdots \lambda_{n-1} \circ (d_X)^2 \circ \lambda_n)\\
&= W \cdot \str_S( \psi \circ \phi \circ \lambda_1 \cdots \lambda_n)\,.
\end{align*}
If we insert this into the residue then the $W$'s cancel as before, and from (\ref{eq:almostfinalklres}) we conclude
\begin{align*}
\langle \psi, \phi \rangle &= \epsilon \Res_{S/k}\! \begin{bmatrix} \str_S\left( \psi \circ \phi \circ \lambda_1 \cdots \lambda_n \right) \\ w'_1, \ldots, w'_n \end{bmatrix}\\
&= \frac{1}{n!} \epsilon \sum_{\sigma \in S_n} \textup{sgn}(\sigma) \Res_{S/k}\! \begin{bmatrix} \str_S\left( \psi \circ \phi \circ \lambda_{\sigma(1)} \cdots \lambda_{\sigma(n)} \right) \\ w'_1, \ldots, w'_n \end{bmatrix}\\
&= \frac{1}{n!} \epsilon \Res_{S/k}\! \begin{bmatrix} \det(C) \cdot \str_S\left( \psi \circ \phi \circ \ud_{S/k}(d_X)^{\wedge n} \right) \\ w'_1, \ldots, w'_n \end{bmatrix}\\
&= \frac{1}{n!} \epsilon \Res_{S/k}\! \begin{bmatrix} \str_S\left( \psi \circ \phi \circ \ud_{S/k}(d_X)^{\wedge n} \right) \\ w_1, \ldots, w_n \end{bmatrix}
\end{align*}
In the first step we use the fact that the pairing $\langle -, - \rangle$ is independent of the ordering of the regular system of parameters and homotopies, so that effectively the $\lambda_i$'s anticommute within the supertrace (see Remark \ref{remark:gradedcomm} and also \cite[Appendix A]{dm1102.2957} for a direct proof) and therefore all permutations contribute equally. In the last step we use the transformation rule.
\end{proof}

Consider for any matrix factorisation $X$ the trace map
\begin{equation}\label{eq:tracemapklfinal1}
\langle - \rangle: \m(X, X[n]) \lto k
\end{equation}
defined by $\langle \psi \rangle = \langle \psi, 1 \rangle$, that is
\begin{equation}\label{eq:tracemapklfinal2}
\langle \psi \rangle = \frac{1}{n!}(-1)^{\binom{n-1}{2}} \Res_{S/k}\!\begin{bmatrix} \str_S\!\left(\psi \circ \ud_{S/k}(d_X)^{\wedge n}\right) \\ \partial_1 W, \ldots, \partial_n W \end{bmatrix}\,.
\end{equation}
The nondegenerate pairing $\langle -, - \rangle$ is determined by the trace map, since $\langle \psi, \phi \rangle = \langle \psi \circ \phi \rangle$. If we want to emphasise the underlying matrix factorisation, we write $\langle - \rangle_X$ for $\langle - \rangle$. For convenience, let us state the following immediate consequence of Lemma \ref{lemma:ptrproperties}:

\begin{lemma} For matrix factorisations $X,Y$ we have:
\begin{itemize}
\item[(i)] For morphisms $\psi: Y \lto X[d-1]$ and $\phi: X \lto Y$, $\langle \psi \circ \phi \rangle_X = \langle \phi \circ \psi \rangle_Y$.
\item[(ii)] For a morphism $\psi: X \lto X[d-1]$, $\langle \psi \rangle_X = (-1)^d \cdot \langle \psi \rangle_{X[1]}$.
\end{itemize}
\end{lemma} 

\begin{remark} There is a more general statement which follows from the lemma: if $\psi: Y \lto X$ is a morphism of degree $d-1-a$ and $\phi: X \lto Y$ is a morphism of degree $a$ then
\[
\langle \psi \circ \phi \rangle_X = \langle \phi \circ \psi \rangle_{Y[a]} = (-1)^{da} \langle \phi \circ \psi \rangle_Y = (-1)^{|\psi| |\phi|} \langle \phi \circ \psi \rangle_Y\,.
\]
\end{remark}

\begin{remark}\label{remark:boundary_bulk_hypersurface} A classical invariant associated to the singular hypersurface $\{ W = 0 \}$ is the $k$-algebra $\Omega_W = S/(\bs{w})$, called the \emph{Jacobi algebra} of $W$. It is a classical result of local duality (see for example \cite[p.659]{Griffiths} when $k = \mathbb{C}$) that $\Omega_W$ together with the functional
\begin{gather*}
\gamma: \Omega_W \lto k,\\
\gamma(s) = \Res_{S/k}\! \big[ s \cdot \ud V \;/\; \partial_{1} W, \ldots, \partial_{n} W \big]
\end{gather*}
is a Frobenius algebra, that is, the pairing $(r,s) = \gamma(rs)$ is nondegenerate. The explicit formula suggests that the trace map of (\ref{eq:tracemapklfinal1}) factors into two pieces
\[
\xymatrix{
\m(X, X[n]) \ar[r]^-{\beta} & \Omega_W \ar[r]^-{\gamma} & k
}
\]
where
\begin{equation}\label{eq:endofbeta}
\beta(\psi) = \frac{1}{n!}(-1)^{\binom{n-1}{2}} \str_S( \psi \cdot \ud_{S/k}(d_X)^{\wedge n})\,.
\end{equation}
This is known in the physics literature as the \emph{boundary-bulk} map, as it sends boundary states (endomorphisms of $X$) to closed states (elements of $\Omega_W$). Rather than argue directly that $\beta$ is well-defined, let us proceed as follows: local duality states that there is an isomorphism of $\Omega_W$-modules
\begin{equation}\label{eq:bulk_boundary_locald_2}
\begin{split}
&\Omega_W \lto \Hom_S( \Omega_W, H^n_{\mf{m}}(S) )\\
r \mapsto &\left\{ s \mapsto \big[ rs \cdot \ud V \;/\; \partial_{1} W, \ldots, \partial_{n} W \big] \right\}.
\end{split}
\end{equation}
The morphism spaces in $\m$ are annihilated by $(\bs{w})$ and the $S$-linear map 
\[
\tau \circ \plangle - \prangle: \m(X, X[n]) \lto H^{n-1}_{\mf{m}}(R) \lto H^n_{\mf{m}}(S)
\]
(with $\tau$ as in Proposition \ref{prop:tauinclusion}) must therefore factor through the submodule $\Hom_S(\Omega_W, H^n_{\mf{m}}(S))$ of $H^n_{\mf{m}}(S)$. Composing this factorisation with the isomorphism (\ref{eq:bulk_boundary_locald_2}) we have a canonical morphism of $\Omega_W$-modules $\m(X, X[n]) \lto \Omega_W$ and this is precisely the map $\beta$ described in (\ref{eq:endofbeta}).
\end{remark}

\begin{example} Set $S = k[[x,y]]$ and $W = x^2 y + y^4$. This is an isolated singularity of type $D_5$. Given a power series $g(x,y)$ we write $C_{x^i y^j}(g)$ for the coefficient of $x^i y^j$ in $g$, or in terms of residue symbols
\[
C_{x^iy^j}(g) = \Res\big[ g \cdot \ud x \land \ud y \,/\, x^{i+1} y^{j+1} \big]\,.
\]
Using the transformation rule as explained in Example \ref{example:residuespowerseries} we see that
\begin{equation}\label{eq:how_to_compute_res_2}
\begin{split}
\Res\Big[\begin{matrix} f \cdot \ud V \;/\; \partial_x W, \partial_y W \end{matrix}\Big] &= \Res\Big[\begin{matrix} (-\frac{1}{2}y^3 + \frac{1}{8}x^2) f \cdot \ud V \;/\; x^3, y^4 \end{matrix}\Big] = -\tfrac{1}{2}C_{x^2}(f) + \tfrac{1}{8}C_{y^3}(f).
\end{split}
\end{equation}
Let $\psi$ be an endomorphism of the matrix factorisation $X$ with differential
\[
d_X = \begin{pmatrix} 0 & d^1 \\ d^0 & 0 \end{pmatrix}, \quad d^0 = \begin{pmatrix} xy & y^2 \\ y^2 & -x \end{pmatrix}, \quad d^1 = \begin{pmatrix} x & y^2 \\ y^2 & -xy \end{pmatrix}\,.
\]
Let us calculate $\langle \psi \rangle$ using the formula (\ref{eq:tracemapklfinal2}). One checks that
\begin{align*}
\str\Big( \psi \left( \partial_x(d_X) \cdot \partial_y(d_X) - \partial_y(d_X) \cdot \partial_x (d_X)\right)\Big) &= x \cdot \psi^0_{11} - 4y^2 \cdot \psi^0_{12} + 4y \cdot \psi^0_{21} - x \cdot \psi^0_{22}\\
&\qquad + x \cdot \psi^1_{11} + 4y \cdot \psi^1_{12} - 4y^2 \cdot \psi^1_{21} - x \cdot \psi^1_{22}\,.
\end{align*}
Hence
\begin{equation}\label{eq:example_trace2}
\begin{split}
\langle \psi \rangle &= \tfrac{1}{4}\Big\{ - C_x(\psi^0_{11}) - C_y(\psi^0_{12}) + C_{y^2}(\psi^0_{21}) + C_x(\psi^0_{22})\\
&\qquad - C_x(\psi^1_{11}) + C_{y^2}(\psi^1_{12}) - C_y(\psi^1_{21}) + C_x(\psi^1_{22}) \Big\}.
\end{split}
\end{equation}
For example, $(\psi^0,\psi^1) = \left( \left(\begin{smallmatrix} 0 & 1 \\ -y & 0 \end{smallmatrix}\right), \left(\begin{smallmatrix} 0 & - y \\ 1 & 0 \end{smallmatrix}\right) \right)$ is an endomorphism of $X$ and we compute that $\langle \psi \rangle = 0$ and $\langle y \cdot \psi \rangle = -1$. The matrix factorisation $X$ is taken from \cite[Ch.9]{Yoshino90}. For further examples of the Kapustin-Li formula in the physics literature, see \cite{Kapustin03, KapustinLi, Herbst05}.
\end{example}

\appendix

\section{Residual complexes and generalised fractions}\label{section:resid_genfrac}

Let $(S,\mf{n},k)$ be a Cohen-Macaulay local ring of dimension $n \ge 1$, and let $(R,\mf{m},k)$ denote the local ring $R := S/W$ for some regular element $W \in S$. Our aim in this appendix is to compare generalised fractions over $S$ and its quotient $R$, using a theorem of Sastry and Yekutieli from \cite{Yek95}. Note that $R$ is Cohen-Macaulay, so both $R$ and $S$ are equidimensional and catenary, and the remarks of \cite[\S 2.1 - \S 2.2]{Yek95} apply. We allow $n = 1$, so that $R$ may be Artinian.

First we recall some basic material from \cite[\S 3.3]{Herzog}. A finitely generated $S$-module $C$ is called a \emph{canonical module} of $S$ if $\dim_k \Ext^i_S(k, C) = \delta_{in}$. A canonical module of $S$ exists if and only if $S$ is a homomorphic image of a Gorenstein local ring and, if a canonical module of $S$ exists, it is unique up to (non-canonical) isomorphism. If $S$ is Gorenstein then $S$ itself is a canonical module, so any canonical module is free of rank one.

Suppose that $S$ has canonical module $\omega_S$ and let $\eta: \omega_S \lto I_S$ be a minimal injective resolution. The complex $I_S$ is then a \emph{residual complex} for $S$, that is, $I_S$ is a bounded below complex of injective $S$-modules with finitely generated cohomology, such that there is an isomorphism
\[
\bigoplus_{i \in \mathbb{Z}} I_S^i \cong \bigoplus_{\mf{p} \in \Spec(S)} E_S(S/\mf{p}).
\]
Indeed, given $\mf{p} \in \Spec(S)$ with $\hht(\mf{p}) = c$ we define $I_S(\mf{p}) := H^c_{\mf{p}}(I_S)$ to be the submodule of elements in $I_S^c$ annihilated by some power of $\mf{p}$. Then $I_S(\mf{p}) \cong E_S(S/\mf{p})$ and for $0 \le c \le n$ the inclusions define a coproduct $I_S^c = \bigoplus_{\hht(\mf{p}) = c} I_S(\mf{p})$ and we write $\mu_{\mf{p}}: I_S^c \lto I_S(\mf{p})$ for the corresponding projection morphisms. We introduce a family of morphisms $\partial_{[t]}$, following Sastry and Yekutieli \cite[\S 2.1]{Yek95}.

\begin{definition} Given a \emph{saturated chain} $(\mf{p}, \mf{q})$ in $\Spec(S)$ (i.e. $\hht(\mf{q}/\mf{p}) = 1$), $\partial_{(\mf{p}, \mf{q})}$ is the morphism
\[
\partial_{(\mf{p}, \mf{q})}: I_S(\mf{p}) \xlto{\inc} I_S^c \xlto{\partial^c} I_S^{c+1} \xlto{\mu_{\mf{q}}} I_S(\mf{q})
\]
where $c = \hht(\mf{p})$. Given $0 \le c < n$ and $t \in S$ we define a morphism $\partial_{[t]}: I_S^c \lto I_S^{c+1}$ by
\[
\partial_{[t]} = \sum_{\substack{(\mf{p}, \mf{q}) \text{ saturated}\\ \hht(\mf{p}) = c, \; t \in \mf{q} \setminus \mf{p}}} \partial_{(\mf{p}, \mf{q})}
\]
\end{definition}

 If $\mf{p}$ is a prime ideal of $S$ then $I_S(\mf{p}) \cong E_S(S/\mf{p})$ so $(0:_{I_S(\mf{p})} W)$ is nonzero if and only if $\mf{p}$ contains $W$, in which case $(0:_{I_S(\mf{p})} W) \cong E_R(R/\mf{p})$ as $R$-modules. In what follows $\hht(\mf{p})$ always denotes the height of the prime ideal $\mf{p}$ in $S$. From the canonical module and its resolution over $S$, we obtain the same data over $R$:

\begin{definition} The module $\omega_R := \omega_S/W\omega_S$ is a canonical module for $R$ with minimal injective resolution $I_R := \Hom_S(R, I_S)[1]$ (see Lemma \ref{lemma:explicit_resmap_kappa} below). Given $\mf{p} \in \Spec(S)$ containing $W$, we set $I_R(\mf{p}) = (0 :_{I_S(\mf{p})} W)$. Then $I_R$ is concentrated in degrees $[0,n-1]$ and given in degree $c$ by
\[
I_R^c = \Hom_S(R, I_S^{c+1}) \cong \bigoplus_{\substack{\hht(\mf{p})=c+1,\\ W \in \mf{p}}} I_R(\mf{p}).
\]
\end{definition}

\begin{lemma}\label{lemma:explicit_resmap_kappa} There is a quasi-isomorphism $\kappa: \omega_R \lto I_R$ defined for $\gamma \in \omega_S$ by
\begin{equation}\label{eq:explicit_resmap_kappa_0}
\kappa(\overline{\gamma}) = \partial_{[W]}\left(\frac{1}{W} \cdot \eta(\gamma) \right) = \sum_{\substack{(\mf{p}, \mf{q}) \text{ saturated}\\ \hht(\mf{p}) = 0, \; W \in \mf{q} \setminus \mf{p}}} \partial_{(\mf{p}, \mf{q})}\left( \frac{1}{W} \cdot \eta(\gamma) \right),
\end{equation}
and this is a minimal injective resolution of $\omega_R$ over $R$.
\end{lemma}
\begin{proof}
Since $W$ is regular it does not belong to any prime ideal $\mf{p}$ in $S$ of height zero, so $\frac{1}{W} \cdot \eta(\gamma)$ makes sense as an element of $I_S^0$. Let $\{ \mf{p}_1, \ldots, \mf{p}_r \}$ be the associated primes of $W$ in $S$, which agree with the minimal primes since $S$ is Cohen-Macaulay. As an $S$-module $R$ has free resolution $P_R$
\[
0 \lto S \xlto{W} S \lto R \lto 0,
\]
and there is a pair of quasi-isomorphisms
\begin{equation}\label{eq:explicit_resmap_kappa_1}
\xymatrix@C+2pc{
\Hom_S(P_R, \omega_S) \ar[r]^-{\Hom(1,\eta)} & \Hom_S(P_R, I_S) & \Hom_S(R, I_S). \ar[l]
}
\end{equation}
The only nonzero cohomology of $\Hom_S(P_R, \omega_S)$ is $H^1\Hom_S(P_R, \omega_S) = \omega_R$, so $\Hom_S(R, I_S)[1]$ is an injective resolution of $\omega_R$ over $R$, with resolution map $\kappa$ obtained by applying $H^1$ to (\ref{eq:explicit_resmap_kappa_1}):
\[
\kappa: \omega_R \cong H^1 \Hom_S(P_R, \omega_S) \xlto{\cong} H^1 \Hom_S(P_R, I_S) \xlto{\cong} H^1\Hom_S(R, I_S) = H^0 I_R.
\]
It remains to calculate $\kappa$. 

Let $\gamma \in \omega_S$ be given and consider the coboundary in $\Hom_S(P_R, I_S)^1$:
\[
\partial^0_{ \Hom_S(P_R, I_S) }( 1/W \cdot \eta(\gamma)) = \begin{pmatrix} -W \\ \partial^0_{I_S} \end{pmatrix}( 1/W \cdot \eta(\gamma) ) = \begin{pmatrix} -\eta(\gamma)\\ \partial^0_{I_S}(1/W \cdot \eta(\gamma)) \end{pmatrix}.
\]
For a prime ideal $\mf{p}$ of height zero in $S$, let $\eta_{\mf{p}}$ denote the composite of $\eta: \omega_S \lto I_S^0$ with the projection $\mu_{\mf{p}}: I_S^0 \lto I_S(\mf{p})$. Since $\eta$ is a morphism of complexes $0 = \mu_{\mf{q}} \partial^0_{I_S} \eta = \sum_{\hht(\mf{p}) = 0} \partial_{(\mf{p},\mf{q})} \eta_{\mf{p}}$ whenever $\hht(\mf{q}) = 1$. Let $\mf{q}$ be a prime ideal of height one not equal to some $\mf{p}_i$, and thus not containing $W$. Then $W$ acts as a unit on $I_S(\mf{q})$ and
\[
\mu_{\mf{q}} \partial^0_{I_S}(1/W \cdot \eta(\gamma)) = \sum_{\hht(\mf{p}) = 0} \partial_{(\mf{p}, \mf{q})} \mu_{\mf{p}} ( 1/W \cdot \eta(\gamma) ) = 1/W \cdot \sum_{\hht(\mf{p}) = 0} \partial_{(\mf{p}, \mf{q})} \eta_{\mf{p}}(\gamma) = 0.
\]
It follows that
\[
\partial^0_{I_S}(1/W \cdot \eta(\gamma)) = \sum_{\hht(\mf{p}) = 0} \sum_{i=1}^r \partial_{(\mf{p},\mf{p}_i)}( 1/W \cdot \eta(\gamma) ) = \partial_{[W]}(1/W \cdot \eta(\gamma)).
\]
Observe that this element is killed by $W$, and therefore belongs to $I_R^0 = \Hom_S(R,I_S)^1$. Hence
\[
\begin{pmatrix} -W \\ \partial^0_{I_S} \end{pmatrix}( 1/W \cdot \eta(\gamma) ) =  \begin{pmatrix} 0 \\ \partial_{[W]}(1/W \cdot \eta(\gamma)) \end{pmatrix} - \begin{pmatrix} \eta(\gamma) \\ 0 \end{pmatrix},
\]
which shows that as we pass from left to right in (\ref{eq:explicit_resmap_kappa_1}) with the cohomology class of $\gamma$, we arrive at $\partial_{[W]}(1/W \cdot \eta(\gamma))$ on the right hand side. We conclude that $\kappa$ is an injective resolution of $\omega_R$, and it only remains to check that $I_R$ is minimal. There exists a minimal subcomplex $J$ of $I_R$, which must be a residual complex since it is a minimal resolution of $\omega_R$. We already know that $I_R$ is a residual complex, so we conclude that $J = I_R$ and $I_R$ is minimal.
\end{proof}

\begin{remark} Since $I_R$ is a residual complex, as above we define morphisms $\partial_{(\mf{p},\mf{q}),R}: I_R(\mf{p}) \lto I_R(\mf{q})$ for any saturated chain $(\mf{p}, \mf{q})$ in $\Spec(R)$ with $\hht(\mf{p}) = c + 1$ (in $S$) and thus for any $t \in R$ a morphism $\partial_{[t],R}: I_R^c \lto I_R^{c+1}$. We also set $H^n_{\mf{n}}(\omega_S) := (I_S)^n$ and $H^{n-1}_{\mf{m}}(\omega_R) := (I_R)^{n-1} = (0 :_{H^n_{\mf{n}}(\omega_S)} W)$.
\end{remark}

This model for local cohomology differs from the one introduced in Section \ref{section:cinjres} and we want to explain how they are related. 

\begin{remark} Let $(B,\mf{m},k)$ be a Cohen-Macaulay local ring of dimension $d \ge 0$ with a canonical module $\omega_B$ and let $\kappa: \omega_B \lto I_B$ be a minimal injective resolution, so $(I_B)^d$ is isomorphic to the injective envelope $E_B(k)$. Let $\bs{t}$ be a system of parameters for $B$. We define
\[
H^d_{\mf{m}}(\omega_B,\bs{t}) := H^d( \skos_{\infty}(\bs{t}) \otimes \omega_B)\,.
\]
Obviously this depends on the system of parameters $\bs{t}$. Let us also set
\[
H^d_{\mf{m}}(\omega_B) := (I_B)^d = \varinjlim_j \Ext^d_B(B/\mf{m}^j,\omega_B).
\]
The injective resolution $I_B$ determines an isomorphism $H^d_{\mf{m}}(\omega_B, \bs{t}) \cong H^d_{\mf{m}}(\omega_B)$, as follows. Tensoring the augmentation $\varepsilon: \skos_\infty(\bs{t}) \lto B$ with $H^d_{\mf{m}}(\omega_B)$ gives an isomorphism
\begin{equation}\label{eq:part_of_local00}
\varepsilon \otimes 1: \skos_{\infty}(\bs{t}) \otimes H^d_{\mf{m}}(\omega_B) \xlto{\cong} H^d_{\mf{m}}(\omega_B)
\end{equation}
since $H^d_{\mf{m}}(\omega_B) \otimes B[t_i^{-1}] = 0$. Let $v: H^d_{\mf{m}}(\omega_B) \lto I_B[d]$ denote the morphism of complexes given by the identity in degree zero. There is a degree-wise split exact sequence
\[
0 \lto H^d_{\mf{m}}(\omega_B) \xlto{v} I_B[d] \lto \Coker(v) \lto 0,
\]
and since $\omega_B$ is a canonical module, $\Coker(v)$ involves only indecomposable injectives $E_B(B/\mf{p})$ for non-maximal primes $\mf{p}$. It follows that $\skos_{\infty}(\bs{t}) \otimes \Coker(v)$ is contractible, and we deduce a homotopy equivalence
\begin{equation}\label{eq:part_of_local}
\xymatrix@C+0.5pc{
H^d_{\mf{m}}(\omega_B) \ar[r]^(0.38){(\ref{eq:part_of_local00})}_(0.38){\cong} & \skos_{\infty}(\bs{t}) \otimes H^d_{\mf{m}}(\omega_B) \ar[r]^(0.5){1 \otimes v} & \skos_{\infty}(\bs{t}) \otimes I_B[d] \ar[r]_{\cong} & ( K_{\infty}(\bs{t}) \otimes I_B)[d]\,.
}
\end{equation}
Hence the resolution morphism $\kappa: \omega_B \lto I_B$ induces the desired isomorphism
\[
\xymatrix@C+2pc{
H^d_{\mf{m}}(\omega_B,\bs{t}) = H^d( \skos_{\infty}(\bs{t}) \otimes \omega_B) \ar[r]^(0.58){H^d(1 \otimes \kappa)}_(0.58){\cong} & H^d( \skos_{\infty}(\bs{t}) \otimes I_B ) \ar[r]^(0.6){(\ref{eq:part_of_local})}_(0.6){\cong} & H^d_{\mf{m}}(\omega_B)\,.
}
\]
We define generalised fractions in $H^d_{\mf{m}}(\omega_B)$ by transferring the generalised fractions in $H^d_{\mf{m}}(\omega_B, \bs{t})$ defined in Section \ref{section:cinjres} along this isomorphism.
\end{remark}

Let $t_2,\ldots,t_n$ denote a system of parameters for $R$ so that $W, t_2, \ldots, t_n$ is a system of parameters for $S$. As explained above, we introduce generalised fractions in $H^{n-1}_{\mf{m}}(\omega_R)$ and $H^n_{\mf{n}}(\omega_S)$, and a theorem of Sastry-Yekutieli relates these generalised fractions to the morphisms $\partial_{[t]}$. 

\begin{theorem}\label{theorem:sastry_yek}(\cite[(2.2.2)]{Yek95}) If $s_1,\ldots,s_n$ is a system of parameters for $S$ and $\gamma \in \omega_S$ then
\[
\big[ \gamma \,/\, s_1,\ldots,s_n \big]_S = (-1)^{\binom{n}{2}} \partial_{[s_n]} \circ \cdots \circ \partial_{[s_1]} \big( \eta(\gamma) / {s_1 \cdots s_n} \big).
\]
\end{theorem}

Using this theorem one can relate generalised fractions over $R$ and $S$. 

\begin{lemma}\label{lemma:genfracquotientsrelate} For $\gamma \in \omega_S$, $\big[ \overline{\gamma} \,/\, t_2,\ldots,t_n \big]_R = \big[ \gamma \,/\, W, t_2,\ldots,t_n \big]_S$ as elements of $H^{n-1}_{\mf{m}}(\omega_R) \subseteq H^n_{\mf{n}}(\omega_S)$.
\end{lemma}
\begin{proof}
By definition for any $0 \le c \le n - 2$ and $t \in S$ the diagram
\[
\xymatrix@C+2pc@R-0.5pc{
I_R^c \ar[d]_j \ar[r]^{-\partial_{[t],R}} & I_R^{c+1} \ar[d]^{j}\\
I_S^{c+1} \ar[r]_{\partial_{[t]}} & I_S^{c+2}
}
\]
commutes, where $j$ denotes inclusions. In particular by Theorem \ref{theorem:sastry_yek} we have
\begin{align*}
j\big[ \overline{\gamma} \,/\, t_2,\ldots,t_n \big]_R &= (-1)^{\binom{n-1}{2}} j \circ \partial_{[t_n],R} \circ \cdots \circ \partial_{[t_2],R} \big( \kappa(\overline{\gamma})/{t_2 \cdots t_n}\big)\\
&= (-1)^{\binom{n-1}{2}+n+1} \partial_{[t_n]} \circ \cdots \circ \partial_{[t_2]} \circ j\big( \kappa(\overline{\gamma})/{t_2 \cdots t_n}\big)
\end{align*}
which by Lemma \ref{lemma:explicit_resmap_kappa} and Theorem \ref{theorem:sastry_yek} becomes
\begin{align*}
&= (-1)^{\binom{n-1}{2}+n+1} \partial_{[t_n]} \circ \cdots \circ \partial_{[t_2]} \circ \partial_{[W]}\big( \eta(\gamma)/{W t_2 \cdots t_n}\big)\\
&= (-1)^{\binom{n-1}{2}+\binom{n}{2}+n+1}\big[ \gamma \,/\, W, t_2,\ldots,t_n \big]_S = \big[ \gamma \,/\, W, t_2,\ldots,t_n \big]_S
\end{align*}
as required.
\end{proof}

\bibliographystyle{amsalpha}

\end{document}